\documentclass[11pt]{amsart}
\usepackage{amsmath,amssymb}
\usepackage{graphicx}

\usepackage{amsfonts,amscd,latexsym,epsfig, psfrag}

\usepackage{hyperref}
\usepackage[all]{xy}

\newcommand{\er}{{\Diamond}}

\newcommand{\Op}{{\rm Op}}

\newcommand{\cl}{{\rm cl}}

\newcommand{\s}{{\mathfrak s}}

\newcommand{\Tx}{{\Tilde x}}

  \newcommand{\uU}{{\underline{U}}}

\newcommand{\Aa}{{\mathcal A}}

\newcommand{\Mor}{{\rm Mor}}

\newcommand{\Obj}{{\rm Obj}}

\newcommand{\Stab}{{\rm Stab}}

\newcommand{\Tf}{{\Tilde f}}

\newcommand{\Kk}{{\mathcal K}}

\newcommand{\im}{{\rm im\,}}
\newcommand{\less}{{\smallsetminus}}

\newcommand{\p}{{\partial}}
\newcommand{\al}{{\alpha}}

\newcommand{\be}{{\beta}}

\newcommand{\de}{{\delta}}

\newcommand{\ga}{{\gamma}}
\newcommand{\Ga}{{\Gamma}}
\newcommand{\io}{{\iota}}

\newcommand{\la}{{\lambda}}
\newcommand{\La}{{\Lambda}}
\newcommand{\si}{{\sigma}}

\newcommand{\Br}{{\rm Br}}
\newcommand{\Fr}{{\rm Fr}}

\newcommand{\Bb}{{\mathcal B}}

\newcommand{\ov}{\overline}

\newcommand{\id}{{\rm id}}

\newcommand{\rT}{{\rm T}}

\renewcommand{\Tilde}{\widetilde}

\newcommand{\TV}{{\Tilde V}}

\newcommand{\Tsi}{{\Tilde \si}}

\newcommand{\Ii}{{\mathcal I}}

\newcommand{\Vv}{{\mathcal V}}

\newcommand{\Q}{{\mathbb Q}}
\newcommand{\R}{{\mathbb R}}

\newcommand{\Z}{{\mathbb Z}}

\newcommand{\Pp}{{\mathcal P}}
\newcommand{\Hh}{{\mathcal H}}

\newcommand{\SSS}{{\smallskip}}

\newcommand{\bB}{{\bf B}}

\newcommand{\bC}{{\bf C}}
\newcommand{\bG}{{\bf G}}
\newcommand{\bM}{{\bf M}}

\newcommand{\bE}{{\bf E}}
\newcommand{\bK}{{\bf K}}
\newcommand{\bV}{{\bf V}}

\newcommand{\bX}{{\bf X}}

\newcommand{\bZ}{{\bf Z}}

\newtheorem{theorem}{Theorem}[section]

\newtheorem{lemma}[theorem]{Lemma}

\newtheorem{prop}[theorem]{Proposition}

\newtheorem{defn}[theorem]{Definition}
\newtheorem{example}[theorem]{Example}

\newtheorem{remark}[theorem]{Remark}
\newtheorem{rmk}[theorem]{Remark}

\numberwithin{figure}{section}
\numberwithin{equation}{section}

\newcommand{\MS}{{\medskip}}

\newcommand{\NI}{{\noindent}}

\newcommand{\pr}{{\rm pr}}

   \newcounter{qcounter}

\newenvironment{itemlist}
   { \begin{list} {$\bullet$}
         {  \setlength{\itemsep}{.5ex} \setlength{\leftmargin}{2.5ex} } }
   { \end{list} }

\newcommand\quotient[2]{
        \mathchoice
            {
                \text{\raise1ex\hbox{$#1$}\Big/\lower1ex\hbox{$#2$}}%
            }
            {
                #1\,/\,#2
            }
            {
                #1\,/\,#2
            }
            {
                #1\,/\,#2
            }
    }

\newcommand\quot[2]{
                \text{\raise1ex\hbox{$#1$}/\lower1ex\hbox{$\scriptstyle#2$}}
  }

\newcommand\quo[2]{
                \text{\raise1ex\hbox{$#1\!\!$}/\lower1ex\hbox{$\!\scriptstyle#2$}}
  }

\newcommand\qu[2]{
                \text{\raise.8ex\hbox{$\scriptstyle#1\!$}/\lower.8ex\hbox{$\!\scriptstyle#2$}}
  }

\newcommand\qq[2]{
                \text{\raise.8ex\hbox{$#1\!$}/\lower.8ex\hbox{$#2$}}
}

 \title{Strict orbifold atlases and weighted branched manifolds}
 \author{Dusa McDuff}
  \thanks{partially supported by NSF grant DMS1308669}
\address{Department of Mathematics,
 Barnard College, Columbia University}
\email{dusa@math.columbia.edu}
 \date{June 17, 2015, revised October 21, 2015}

\begin{document}
\maketitle

\begin{abstract}  
This note revisits some of the ideas in \cite{Mbr} on orbifolds and branched manifolds,  showing how the constructions can be simplified by using a version
of the Kuranishi atlases developed by McDuff--Wehrheim.
 We first show that every orbifold has such an atlas, and then use it to obtain an explicit model for the 
 nonsingular resolution of an oriented  orbifold $Y$ (which is a weighted nonsingular groupoid  
 with the same fundamental class as $Y$) and  for the
 Euler class of an oriented orbibundle.
In this approach, instead of appearing as the zero set of a  multivalued section,  the Euler class
is the zero set 
 of a single-valued section of the pullback bundle over the resolution, 
 and hence has the structure of
 a weighted branched manifold in which the weights and branching are canonically defined by the atlas.
\end{abstract}

\section{Introduction}\label{s:orb}

A strict   orbifold atlas is a  special case of the Kuranishi atlases developed in \cite{MW1,MW2,MW3} by McDuff--Wehrheim to provide a framework for
the  construction of the virtual moduli cycle in Gromov--Witten theory.  When specialized to the orbifold case (i.e. all obstruction spaces are trivial), such an atlas
 encapsulates the structure of an \'etale proper (ep for short)
 groupoid in a  way that is well adapted to certain constructions, for example that of  
 the Euler class of an orbibundle.  
 Although in this note we restrict attention to the finite dimensional case, our results 
about abstract orbifolds and their representing groupoids  (such as the construction of orbifold atlases, groupoid completions and reductions) 
apply in any setting in which there is an adequate topological and  analytical framework.  
In particular, as outlined in Remark~\ref{rmk:polyf} one should be able to use these ideas in the polyfold context of Hofer--Wysocki--Zehnder~\cite{HWZ}
to describe the zero set of a transverse perturbation of the canonical section of a Fredholm bundle as a weighted branched manifold.

The first section  defines the notion of a strict orbifold atlas, and gives examples showing how the structure  hidden in the morphisms of a groupoid is made explicit   in the atlas.  Such an atlas  $\Kk$ determines an ep category $\bB_\Kk$, which is not a groupoid because its morphisms are not all invertible.  Our main results are:
\begin{itemlist}\item  Proposition~\ref{prop:groupcomp}:  The category $\bB_\Kk$ has a unique completion to a groupoid with the same space of objects and realization, and hence determines a unique orbifold structure on  the realization $|\bB_\Kk|\cong Y$.
 \item Proposition~\ref{prop:orb}:  Conversely, every paracompact orbifold is the realization of a strict orbifold atlas, that is unique up to commensurability.
\end{itemlist}
%
%
%
%
In \S\ref{s:app}, we
 first use the atlas to construct the nonsingular resolution of an orbifold.  
 This is a weighted  \'etale groupoid with at most one morphism between any two objects,
  that also has a weighting function.  Thus its realization is a weighted  branched  manifold, 
  that, if compact and oriented, carries a fundamental class.  (See  Remark~\ref{rmk:cobordinvar} for a  discussion of further 
cobordism   invariants of weighted branched manifolds.)   We then  construct the Euler class of an 
oriented orbibundle over a compact oriented  base using  a single-valued section of the pullback of the bundle over a resolution 
rather than the more customary  multi-valued section.

\subsection{Definition and examples}

As in Adem--Leida--Ruan~\cite{ALR} and  Moerdijk~\cite{Moe} we take a naive approach to orbifolds, since that suffices for our current purposes.  considering them as equivalence classes of groupoids rather than 
as stacks or $2$-categories as in Lerman~\cite{Ler}.
Thus, we  define orbifolds via the concept of  
 {\bf  ep (\'etale proper) groupoid} $\bG$.  This is  a topological category whose
spaces of objects $\Obj_\bG$ and morphisms $\Mor_\bG$ are smooth manifolds\footnote
{
Manifolds are always assumed to be paracompact.} of some fixed dimension $d$, such that
\begin{itemize}\item  all structural maps (i.e. source $s$, target $t$, identity, composition and inverse)
 are {\bf  \'etale} (i.e. local diffeomorphisms);
and \item
 the map $s\times t: \Mor_\bG\to \Obj_\bG\times \Obj_\bG$
given by taking a morphism to its source and target is {\bf proper} (i.e. the inverse image of a compact set is compact).  
\end{itemize}
The {\bf realization} $|\bG|$ of $\bG$ is the quotient of the space of objects by the equivalence relation 
given by the morphisms: thus $x\sim y \ \Leftrightarrow\ \Mor_\bG(x,y)\ne \emptyset$.
 We denote the quotient map by $\pi_\bG: \Obj_{\bG}\to |\bG|$.
Note that, when (as here) the domains are locally compact, the
 properness condition implies that $|\bG|$ is Hausdorff. 
 We say that $\bG$ is
 \begin{itemize}\item
  {\bf effective} if the only connected components of $\Mor_{\bG}$ on which the source map $s$ equals the target map $t$ 
 consist entirely of identity morphisms;
 \item {\bf nonsingular} if $\Mor_{\bG}(x,y)$ contains at most one element for all $x,y\in \Obj_{\bG}$;
  \item {\bf oriented} if both manifolds  $\Obj_{\bG}$ and $\Mor_{\bG}$ carry an orientation that is preserved by all structural maps. 
 \end{itemize}

For example, if  a finite group $\Ga$  acts smoothly on a smooth manifold $U$ then naively one thinks of the quotient $ \qu{U}{\Ga}$ as an orbifold.
In this situation we define the ep groupoid $\bG_{(U,\Ga)}$ to have  $$
\Obj_\bG= U,\quad \Mor_\bG= U\times \Ga,\quad
(s\times t)(u,\ga)=(\ga^{-1}u, u),
$$
with the obvious identity, inverse  and composition maps. There is a map $f:U\to Y$ (the analog of the footprint map for a Kuranishi chart)
that induces a homeomorphism $f:  |\bG|=\qu{U}{\Ga} \to Y$.  More generally, we make the following definitions.

\begin{defn}\label{def:orbstr}
An {\bf orbifold structure} on a 
paracompact Hausdorff 
 space $Y$ is a pair $(\bG, f)$ consisting of an ep (\'etale proper) groupoid $\bG$ together with a 
map $f:\Obj_\bG\to Y$ that factors through a homeomorphism $|f|:|\bG|\to Y$.  
A {\bf refinement } of $(\bG,f)$ is an orbifold structure $(\bG'', f'')$ on $Y$
 together with a functor $F:  (\bG'', f'')\to (\bG,f)$ such that
\begin{itemize}\item $F$ is
 \'etale (i.e.  the induced maps on objects and morphism spaces are local diffeomorphisms); 
 \item $F$ is full and faithful, i.e.
 $F_*: \Mor_{\bG''}(x,y)\to \Mor_{\bG}\bigl(F(x),F(y)\bigr)$  is an isomorphism for all $x,y\in \Obj_{\bG''}$;
 \item  $f'' = f\circ F$.
  \end{itemize}
  Two orbifold structures $(\bG, f)$ and
$(\bG', f')$ are said to be {\bf Morita equivalent} if they have a common refinement, i.e. if there is a third structure $(\bG'', f'')$ on $Y$ and functors $F:\bG''\to\bG, F':\bG''\to \bG'$ as above.
  An {\bf orbifold} is a 
paracompact
Hausdorff space $Y$ equipped 
with an equivalence class of orbifold structures.  We say that $Y$ is {\bf oriented} if for each representing groupoid $\bG$
 the spaces  $\Obj_\bG$  and  $\Mor_\bG$ have  orientations that are preserved by  all structure maps and by the functors $F:\bG\to \bG'$ considered above.
\end{defn}

\begin{defn}\label{def:atl} A {\bf local chart} 
$(U, \Ga, \psi)$  
on a topological space $Y$ is a triple consisting of
 a connected  open subset $U\subset \R^d$, a finite group 
 $\Ga$ that acts by  diffeomorphisms of $U$
and a map $\psi: U\to Y$ that factors through a homeomorphism from the quotient $\uU: = \qu{U}{\Ga}$ onto an open subset $
F$ of $Y$ called the {\bf footprint}.  

If $Y$ is an orbifold, then in addition we  require
this chart (in this case also called a {\bf  local uniformizer}) to determine the smooth structure of $Y$
over $F$ in the sense that for one (and hence any) orbifold structure $(\bG,f)$ on $Y$ 
each $x\in f^{-1}(F)$ has a neighbourhood $V\subset f^{-1}(F)$ that is locally diffeomorphic to $(U,\Ga)$.  More precisely,
 if $\Ga_\bG^x: = \Mor_{\bG}(x,x)$, resp. $\Ga^x$,  is the stabilizer of $x$ in $\bG$, resp. $\Ga$,
then $f$ lifts to a map $\Tf:V\to U$ that is an embedding (i.e. a diffeomorphism onto its image)
and is such that
\begin{itemize}\item $\Tf$ is equivariant with respect to some isomorphism $\Ga_\bG^x\stackrel{\cong}\to \Ga^x\subset \Ga$ and 
\item  the induced map $\qu{V}{\Ga^x_\bG}\to \qu{U}{\Ga}$  is a homeomorphism to its image;
\item  if $Y$ is  oriented, then we also require $U$ to be oriented compatibly with all the above maps.
\end{itemize}
  \end{defn}

It is well known that every 
orbifold $Y$ has a locally finite covering family of such charts  $\bigl(U_i,\Ga_i,\psi_i \bigr)_{i\in A}$; i.e. we have
$Y = \bigcup_{i\in A} \psi_i(U_i)$ and $\bigcap_{i\in I} \psi_i(U_i) \ne \emptyset \Longrightarrow |I|<\infty$.
 Indeed, given any representing groupoid $(\bG,f)$  Robbin--Salamon~\cite[Lemma~2.10]{RobS}
  construct a covering family  from $\bG$ 
 in the sense that each $U_i$ is a subset of $\Obj_{\bG}$ such that 
 the full subcategory of $\bG$ with objects $U_i$ is isomorphic to the category
 $\bG_{(U_i,\Ga_i)}$ defined above. 
 Although, in this situation the covering family in some sense  generates the groupoid $\bG$, 
 there could be many components in $\Obj_{\bG}$ and $\Mor_\bG$  that we know very little about. 
 We might ask: what is the minimal extra structure  needed to determine the orbifold structure on $Y$?
  
  We will see that 
the following notion gives a simple answer to this question.

\begin{defn}\label{def:orbat}  A 
 {\bf strict orbifold atlas} 
 $\Kk = \bigl(\bK_I, \rho_{IJ}\bigr)_{I\subset J\in \Ii_Y}$ on a paracompact Hausdorff space $Y$ consists of the following data:
\begin{itemlist} 
\item[\rm (i)]
a locally finite  open cover $(F_i)_{i\in A}$ of $Y$, with associated set 
 $\Ii_Y: =  \bigl\{I\subset  A
:F_I: ={\textstyle  \bigcap_{i\in I} }F_i\ne \emptyset\bigr\}
$;
\item[\rm (ii)] a collection $\bigl(W_I,\Ga_I,\psi_I\bigr)_{I\in \Ii_Y}$ of local charts 
where  $\Ga_I: = \prod_{i\in I} \Ga_i$
%
   with footprints $\psi_I(W_I) = F_I$ such that
when $|I|>1$  the group $\Ga_{I\less \{ i\}}$ acts freely on
$W_I$  for each $i\in I$; and
\item[\rm (iii)]  a family of 
 smooth local diffeomorphisms (or covering maps)
$$
\rho_{IJ}: W_J\to W_{IJ}: = (\psi_I)^{-1}(F_J)\subset W_I, \qquad I\subset J,\; I,J\in \Ii_Y,
$$
satisfying the following conditions for all $I\subset J, \; I,J\in \Ii_Y$:
\begin{itemize} \item[(a)]  $\rho_{JJ} = \id$;
 \item[(b)]  if $I\subsetneq J$ then  $\rho_{IJ}$ is 
 equivariant with respect to the projection
$\rho_{IJ}^\Ga: \Ga_J\to \Ga_I$, and is given by  
 the composite of the quotient of $W_J$ by the free action of 
$\Ga_{J\less I}$  with a $\Ga_I$-equivariant diffeomorphism $\qu{W_J}{\Ga_{J\less I}}\to W_{IJ}\subset W_I$;
\item[(c)]   $\psi_I\circ \rho_{IJ} = \psi_J$, and $\rho_{IJ}\circ \rho_{JK} = \rho_{IK}$ for all $I\subset J\subset K$.
\end{itemize}
\end{itemlist}
The {\bf charts} of this atlas $\Kk$ are the tuples $\bigl(\bK_I: = (W_I,\Ga_I,\psi_I)\bigr)_{I\in \Ii_Y}$ with {\bf footprints }
$(F_I)_{I\in \Ii_Y}$ and {\bf footprint maps} $\psi_I$, and the {\bf coordinate changes} are induced by the covering maps $\rho_{IJ}$.
\end{defn}

It is often useful to think of the charts $(\bK_i:=\bK_{\{i\}})_{i\in A}$ as the {\bf basic charts}, while the $\bK_I$ with $|I|>1$ are {\bf transition charts} that define how the basic charts fit together. For short, we will often call an atlas with the above properties 
an orbifold atlas.  
\footnote
{
  We warn the reader that an orbifold atlas (or good atlas) is customarily defined to be
 a covering family of charts that satisfy a somewhat different compatibility condition on overlaps; see 
for example \cite{ALR,MPr, Mbr}.
}

\begin{rmk}\rm (i) It is not hard to check that the projections $(\rho_{IJ}, \rho_{IJ}^\Ga):W_J\to W_{IJ}$ (which are called {\bf group coverings} in \cite[\S2.1]{MW3}) induce isomorphisms on the stabilizer 
subgroups, i.e. if $x = \rho_{IJ}(y)$ then $ \rho^\Ga_{IJ}: \Ga_J^y\stackrel{\cong}\to \Ga_I^x$.\MS

\NI (ii)  By slight abuse of language, we often call the group $\Ga_I$  the {\bf isotropy group} of the chart $\bK_I$, even though 
in general it does not  equal the stabilizer subgroup $\Ga_I^x$  of any point $x\in W_I$.   
Although one could insist that the basic charts $(W_i,\Ga_i, \psi_i)$ are {\bf minimal} in the sense that $\Ga_i = \Ga_i^x$ for some $x\in W_i$, this property is not preserved by arbitrary  restrictions to $\Ga_i$-invariant subsets of $W_i$  and also, because the groups $\Ga_{J\less I}$ act freely,  will usually not hold for the transition charts.  
One can think of  $\Ga_I$ as the automorphism group (or stabilizer) of the footprint map $\psi_I: W_I\to Y$ 
  in an appropriately defined category of \lq\lq stacky" maps $(W,\psi)$ from manifolds $W$ to the orbifold $Y$.
\hfill$\er$
\end{rmk}

As we show in Proposition~\ref{prop:groupcomp} below, the above notion of atlas on the topological space $Y$  is sufficient to
give a complete description of its  structure as an orbifold.  In particular,
as in \cite[Definition~2.3.5]{MW3}, each such  atlas $\Kk$  defines a category $\bB_\Kk$ with $\Obj_{\bB_\Kk} = \bigsqcup_{I\in \Ii_Y} W_I$ and
morphisms  $\Mor_{\bB_\Kk} = \bigsqcup_{I\subset J, I,J\in \Ii_Y} W_J\times \Ga_I$, with source and target given by
\footnote
{
In hindsight, it might have been more natural to consider the tuple $(I,J,y,\gamma)$ as a morphism with source  $(J,y)$ rather than $(I,\ga^{-1}\rho_{IJ}(y))$ since the only way to obtain a smooth parametrization of the morphisms trom $W_I$ to $W_J$ is to parametrize them by the points in $W_J$.   However we will follow the conventions in the papers \cite{MW1,MW2,MW3}. Note also that below  we write compositions in the categorical ordering. 
} 
\begin{equation}\label{eq:Bcomp}
\Ii_Y\times \Ii_Y\times W_J\times \Ga_I\;\ni\;(I,J,y,\gamma)\in \Mor_{\bB_\Kk}\bigl((I,\ga^{-1} \rho_{IJ}(y)), (J, y)\bigr).
\end{equation}
Composition is defined by
\begin{equation}\label{eq:Bcomp1}
\bigl(I,J,y, \ga\bigr)\circ \bigl(J,K,z,\de\bigr)
:= \bigl(I,K,z,
\rho^\Ga_{IJ}(\de) \ga \bigr) \quad \mbox{ if } \de^{-1} \rho_{JK}(z)=y.
\end{equation}

The {\bf realization} $|\bB_\Kk|$ of the category $\bB_\Kk$ is defined to be the quotient $\Obj_{\bB_\Kk}/\!\!\sim$, where $\sim$ is the equivalence relation on objects generated by setting $x\sim y$ whenever $\Mor(x,y)\ne \emptyset$.   The following lemma is a special case of \cite[Lemma~2.3.7]{MW1}.  Its proof is elementary. 

\begin{lemma}  The category $\bB_\Kk$ is well defined; in particular, 
composition  is associative.   It is \'etale and proper.
Moreover, the footprint maps $\psi_I$ induce a homeomorphism  $|\psi|: |\bB_\Kk| \to Y$.
\end{lemma}

\begin{example}\label{ex:triv}\rm {\bf (Manifolds)}  Every manifold\footnote
{
assumed paracompact}
 $Y$ is the realization of the \'etale proper (ep) category 
$\Op(Y)$ with objects equal to the disjoint union $\bigsqcup_{\al\in \Aa} U_\al$ of all open subsets of $Y$ and morphisms
given by inclusion. Thus if $\io_\al:U_\al\to Y$ is the inclusion and we order the elements of $\Aa$ by the inverse inclusion relation
so that $\al\le \be \Longrightarrow \im(\io_\al) \supset \im(\io_\be)$, then
 $\Mor_{\Op(Y)} = \bigsqcup_{\al\le \be}  U_\be$ with source and target given by
 $$
 (\al,\be,x): (\al, \io_\al^{-1} \circ \io_\be(x))\mapsto (\be, x),\quad \al\le \be , \; x\in U_\be.
 $$
Every locally finite open covering  $(W_i)_{i\in A} $ of $Y$ defines an atlas on $Y$ with trivial  isotropy
 groups $\Ga_I$ whose corresponding category $\bB_\Kk$ is a full subcategory of $\Op(Y)$.  
However, Definition~\ref{def:orbat} also allows for atlases on $Y$ with nontrivial  isotropy 
groups $\Ga_I$.  The condition for $Y$ to be a manifold is that all stabilizer subgroups $\Ga_I^x: = \{\ga\in \Ga_I \ | \ \ga(x) = x\}$ of the points $x\in W_I$ are trivial; in other words, each group $\Ga_I$ must act freely on $W_I$ so that the footprint maps $\psi_I: W_I\to Y$ are local homeomorphisms.
Since $\Ga_I : = \prod_{i\in I}\Ga_i$,  the assumptions on the covering maps $\rho_{IJ}$ imply that this will hold for all charts provided 
 that it holds for  the basic charts.  
 \hfill$\er$
\end{example}

\begin{example}\label{ex:foot}\rm  (i)
A first nontrivial example is a ``football" $Y = S^2$ with two basic  charts $(W_1, \Ga_1=\Z_2, \psi_1)$, $(W_2, \Ga_2=\Z_3,\psi_2)$ 
that parametrize neighbourhoods $\psi_i(W_i) = F_i\subset S^2$ of the northern resp.\ southern hemisphere with isotropy of order $2$ resp.\ $3$ at the north resp.\ south pole.
The restrictions of the basic charts to the annulus $F_{12}: =F_1\cap F_2$  have domains given by the annuli $W_{i(12)}: = \psi_i^{-1}(F_{12})$
that each support a free action of the relevant  group $\Ga_i$.
There is no direct functor between these restrictions because the coverings $W_{1(12)}\to F_{12}$ and $W_{2(12)}\to F_{12}$ are incompatible.  
However, they can be related by a common free covering, namely the pullback defined by the diagram
\begin{equation}\label{W12}
\xymatrix{
W_{12}   \ar@{-->}[d] \ar@{-->}[r]   & W_{1(12)} \ar@{->}[d]^{\psi_1}   \\
W_{2(12)} \ar@{->}[r]^{\psi_2}  & Y.
}
\end{equation}
Thus  $W_{12} := \{(x,y)\in W_{1(12)}\times W_{2(12)} \,|\, \psi_1(x) = \psi_2(y)\}$
with group $\Ga_{12}: = \Ga_1\times \Ga_2 = \Z_2 \times \Z_3$.
The corresponding footprint map $\psi_{12}:W_{12}\to F_{12}$ is the $6$-fold covering of the annulus, and the  coordinate changes from $(W_i,\Ga_i,\psi_i)|_{F_{12}}$ to $(W_{12},\Ga_{12},\psi_{12})$ are the coverings $W_{12} \to  W_{i(12)	}$ in the diagram.
Therefore the category $\bB_\Kk$ in this example has index set $\Ii_Y = \{1,2, 12\}$, 
objects the disjoint union $\bigsqcup_{I\in \Ii_Y} W_I$, 
and
morphisms
$$ 
\Bigl({\textstyle  \bigsqcup_{I\in \Ii_Y}} W_I\times \Ga_I  \Bigr) \cup \Bigl({\textstyle  \bigsqcup_{i=1,2} W_{12}\times \Ga_i  }\Bigr) ,
$$
where for $i=1,2$ the elements in $W_{12}\times \Ga_i$ represent the morphisms from $W_{i} $ to $W_{12}$.
\MS

\NI (ii)
The \lq\lq simplest" groupoid $\bG$ with $|\bG| = Y$ would have objects $W_1\sqcup W_2$ and the following  morphisms:
\begin{itemize} \item  morphisms from $W_i$ to itself parametrized by 
$ W_i\times \Ga_i$; 
\item morphisms from $W_1$ to $W_2$ parametrized by $W_{12}$ with $$
s\times t: W_{12}\to W_{1(12)}\times W_{2(12)},
 x\mapsto \bigl(\rho_{1(12)}(x), \rho_{2(12)}(x)\bigr);
 $$
 \item another copy of $W_{12}$ representing the inverses of these morphisms.
\end{itemize}
The fact that this groupoid has such a simple description is a consequence of the existence of the pullback diagram~\eqref{W12}.  
However, even in this case it is not so easy to give an explicit formula for the composition $\Mor(W_1,W_2)\circ \Mor(W_2,W_1) \to \Mor(W_1,W_1) = W_1\times \Ga_1$, which is necessary if one wants to describe a groupoid rather than a category.
In the atlas, the space $W_{12}$ is considered as another component of the object space, which
firstly allows us to order  the components $W_I$ of the object space  so that we need not consider all 
morphisms but only those from $W_I$ to $W_J$ with $I\subset J$, and secondly
allows us to  replace the space of direct  morphisms from $W_1$ to $W_2$ 
by the space of morphisms from $\sqcup_i W_i$ to $W_{12}$, thus decomposing the morphisms from $W_1$ to $W_2$ into constituents
 that are easier  to describe.
 \hfill$\er$
\end{example}

The
simple construction in Example~\ref{ex:foot} does not work for arbitrary orbifolds since the (set theoretic) pullback $W_{12}$ considered above will not be a smooth manifold if any point in 
$\psi_1(W_1)\cap \psi_2(W_2)$ 
has nontrivial stabilizer.  
However, it turns out that there is a very simple substitute construction for the domain of the transition  chart.  Namely,  if the charts $(W_i,\Ga_i)$ inject into a groupoid representative for $Y$ then we can take $W_{12}$ to be the morphisms in this groupoid from $W_1$ to $W_2$; see  the proof of Proposition~\ref{prop:orb}.  This morphism space is the \lq\lq stacky" analog of the fiber product;  see Pardon~\cite[\S2.1.2]{Pard}, who also observes 
that this can be used to construct orbifold atlases.
\MS

\begin{remark}\label{rmk:var} \rm {\bf (Variations on the definition)} In certain geometric situations, such as 
the case of products discussed in Example~\ref{ex:prod} below, it is natural to generalize the definition of atlas
\footnote
{
The requirements below are similar to, but simpler than, the conditions in \cite{Mcn} for a \lq\lq semi-additive atlas": there we also had to take into consideration additivity requirements for the obstruction spaces.}
 to allow
for  the possibility that  the indices  $i\in A$ associated to the footprints $F_i$ and groups $\Ga_i$ of the basic
covering family of $Y$ do not all correspond to local uniformizers of $Y$, though there are enough charts with footprints equal to 
intersections $F_I$ to cover $Y$.
 
Thus we define  a {\bf generalized orbifold atlas} to consist of a locally finite open cover $\{F_i\}_{i\in A}$ of $Y$, a family of finite groups $(\Ga_i)_{i\in A}$, and a subset $\Ii_Y\subset \Pp^*(A)$ of the set  
of finite nonempty subsets of $A$ satisfying the following conditions:
\begin{itemize}\item  $I\in \Ii_Y\Longrightarrow F_I: = \bigcap_{i\in I} F_i\ne \emptyset$;
\item  if $I\in \Ii_Y$ and $ I\subset J$  then $J\in \Ii_Y$ if and only if $F_J\ne \emptyset$;
\item $\bigcup_{I\in \Ii_Y} F_I = Y$;
\item for each $I\in \Ii_Y$ there is a local chart $(W_I,\Ga_I,\psi_I)$ with group $\Ga_I: = \prod_{i\in I}\Ga_i$
and footprint $F_I$;
\item  the family of charts $(W_I,\Ga_I,\psi_I)_{I\in \Ii_Y}$ also satisfy conditions (ii), (iii) in the Definition~\ref{def:orbat} of an orbifold atlas.
\end{itemize}
Pardon's notion of an implicit atlas is yet more general, since he does not insist that the domains of his charts are manifolds. As he explains in 
\cite[Remarks~2.1.3,~2.1.4]{Pard}, his definitions are in some respects simpler.  However,  we need an explicit description of the \'etale category $\bB_\Kk$ in order to be able to perform certain geometric constructions,
such as the construction of a perturbation section in \cite[\S7.3]{MW2}, or the nonsingular resolution below.
\hfill$\er$
\end{remark}

\begin{example}\label{ex:prod}\rm {\bf (Products)} Consider the product $Y= Y_1\times Y_2$  of two orbifolds, where $Y_\al$
is equipped with the
 atlas $\Kk_\al = (W^\al_I,\Ga^\al_I,\psi^\al_I)_{I\in \Ii_{Y_\al}} $ 
 with basic charts indexed by the elements of $A_{\al}$. 
 Then the 
 family of product charts
 $$
 \bigl(W^1_{I_1}\times W^2_{I_2}, \Ga^1_{I_1}\times \Ga^2_{I_2} , \psi^1_{I_1}\times \psi^2_{I_2}\bigr), \quad (I_1,I_2)\in \Ii_{Y_1}\times \Ii_{Y_2}  
 $$
is a generalized atlas indexed by $\Ii_Y\subset \Pp^*(A_1\sqcup A_2)$, where $\Ii_Y: = \{(I_1,I_2): I_\al\in \Ii_{Y_\al}\}$.
   Here 
we take $A $ to be the disjoint union $A_1\sqcup A_2$, 
denoting the elements in $A_1$ by pairs $(i,\emptyset)$ for $i\in A_1$ and those of $A_2$ by $(\emptyset, j)$ for $j\in A_2$. 
If we write the  elements of $\Pp^*(A)$ as  pairs $(I_1,I_2)$, where $I_\al\in \Pp(A_\al)$ are not both empty, then 
$\Ii_Y$ consists of pairs $(I_1,I_2)$ where neither set $I_\al$ is empty.
On the other hand, we can define footprints corresponding to all nonempty subsets of $A$ as follows: define
$$
F_{i,\emptyset}: = F^1_{i}\times Y_2,\;\; i\in A_1\;\;\mbox{  and }\;F_{\emptyset j}: = Y_1\times F^2_{j},\;\; j\in A_2,
$$
and then set
$F_{I_1,I_2} = F^1_{I_1}\times F^2_{I_2} = \bigcap_{i\in I_1} F_{i,\emptyset} \cap \bigcap_{j\in I_2} F^2_{\emptyset,j} 
$
Similarly, we can define $\Ga_{i,\emptyset}: = \Ga^1_i, \Ga_{\emptyset, j}: = \Ga^2_j$, and then define the other $\Ga_{I_1,I_2}$ as products of these groups.

In Pardon's approach, one can
 include a \lq\lq chart" 
that is indexed by the empty set, namely  $(W_\emptyset:=Y,\Ga_{\emptyset} = \id, \psi_{\emptyset} = \id_Y)$, and then include product  charts of the form $Y_1 \times (W_J,\Ga_J,\psi_J)$ as part of the atlas.   
\hfill$\er$
\end{example}

See Example~\ref{ex:gerbe} for a description of some atlases on noneffective orbifolds.

\section{Groupoid completions}

Although $\bB_\Kk$ is not a groupoid since some of the nonidentity maps are not invertible, 
we now show that this category has a canonical {\bf groupoid completion} $\bG_\Kk$. 
(This justifies our language since it implies that any paracompact Hausdorff space $Y$ with an orbifold atlas is in fact an orbifold.)

\begin{defn}\label{def:groupc}  Let $\bM$ be an \'etale proper  category with objects $\bigsqcup_{I\in \Ii} W_I$ 
and realization $Y : = \Obj_\bM/\!\sim$ such that 
\begin{itemize}\item
for each $I\in \Ii$ the full subcategory of $\bM$ with objects $W_I$ 
can be identified with the group quotient $(W_I,\Ga_I)$ for some group $\Ga_I$;
\item for each $I\in \Ii$ the realization map $\pi_\bM:\Obj_\bM \to Y$ induces a homeomorphism
  $\qu{W_I}{\Ga_I}\to F_I\subset Y$, where $F_I$ is an open subset of $Y$.
\end{itemize}
 Then we say that an ep groupoid $\bG$ is a {\bf groupoid completion} of $\bM$ if
 there is an injective functor $\io:\bM\to \bG$ that induces a bijection on objects, an isomorpyhism on stabilizer subgroups,  and 
 a homeomorphism on the realizations $Y= |\bM|\to |\bG|$. 
\end{defn}

Thus 
for each component $W_I$ of $\Obj_\bM$ the groupoid completion (if it exists)
has the same morphisms from $W_I$ to $W_I$ but (unless  $\bM$ is already a groupoid) will have more morphisms between the different components of $\Obj_\bM$ that are obtained by adding inverses and composites.
Before giving the  general construction for $\bG$, we consider the following simple example.

\begin{example}\label{ex:grouppt}\rm  Consider an atlas on the orbifold $Y$ consisting of a single point with stabilizer group $S$ 
with basic  charts  labelled by $\{1,\dots,N\}$ so that  $\Ii_Y$ is
the set of {\it all} subsets of $\{1,\dots,N\}$.  
Each group $\Ga_i$ acts transitively on $W_i$, so that we can identify $W_i\cong  \qu{\Ga_i}{S_i}$, where $S_i$ is the stabilizer of some point $x_i\in W_i$ and $\Ga_i$ acts on the quotient $W_i$ by multiplication on the left $\ga\cdot aS_i = \ga aS_i$.
Similarly  for each $I$, the group $\Ga_I = \prod_{i\in I} \Ga_i$ acts transitively on $W_I$
and we can identify $W_I: =  \qu{\Ga_I}{S_I}$   where $S_I = \Stab_{\Ga_I}(x_I)$.  
The 
 equivariant covering map $(\rho_{iI},\rho^\Ga_{iI}) $ identifies the stabilizer of the point $x_I\in W_I$ with the stabilizer 
 of its  image 
 $\rho_{iI}(x_I) \in W_i$.  Therefore the subgroups $S_I\subset \Ga_I$ can be canonically  identified provided that we can choose a 
 family of base points $x_I\in W_I$ that are consistent in the sense that $\rho_{IJ}(x_J) = x_I$ for all $I\subset J$.
 This is possible because
$\Ii_Y$ has a maximal element  $I_{\max} = \{1,\dots, N\}$.    Thus, we can fix $x_{\max}\in W_{I_{\max}}$ and then define $x_I: = \rho_{I (I_{\max})}(x_{\max})$ for all $I\in \Ii_Y$ so that $\rho_{IJ}(x_J)= x_I$ for all $I\subset J$. This gives
consistent identifications  of $S: = S_{I_{\max}}$ with $S_I: = \Stab_{\Ga_I}(x_I)$ for all $I$. In particular,
we  identify $S$ with the subgroup $ \Stab_{\Ga_i}(x_i)\subset \Ga_I$ for all $i$ so that we may write
$W_I: =   \qu{\Ga_I}{S} $,  where 
$S$ acts diagonally on $\Ga_I$ by $(\ga_{i_1},\dots,\ga_{i_k})\mapsto (\ga_{i_1}s,\dots,\ga_{i_k}s)$.
Thus the category $\bB_S$ corresponding to this atlas has the following description:
$$
\Obj_{\bB_S} = \bigsqcup_{I\subset \Ii_Y}W_I=\qu{\Ga_I}{S}\,,\quad \Mor_{\bB_S} = \bigsqcup_{I\subset J, I,J\in \Ii_Y}\Mor(W_I,W_J) =\qu{\Ga_J}{S} \times \Ga_I,
$$
where,
with $\ga_I: = (\ga_I^i)_{i\in I}\in \Ga_I$, $\ga_IS: = \{(\ga_I^is)_{i\in I}: s\in S\}  \in \qu{\Ga_I}{S}$, 
and abbreviating the projections $\rho_{IJ} $ by restrictions denoted for example as $\ga_J|_I$,
we have
\begin{align*}
& s\times t: \qu{\Ga_J}{S} \times \Ga_I\to \qu{\Ga_I}{S}\times \qu{\Ga_J}{S},\\
& \qquad  \qquad  \qquad 
 (\ga_JS,\de_I)\mapsto \bigl(\de_I^{-1}\rho_{IJ}(\ga_J)S, \ga_JS\bigr) = :  \bigl(\de_I^{-1}\ga_J|_IS, \ga_JS\bigr),
 \end{align*}
 and, when $I\subset J\subset K$, 
 \begin{align*}
& 
m_S\bigl((\ga_JS, \de_I),(\ga_KS, \de_J)\bigr) = 
\bigl(\ga_KS, \de_J|_I\,\de_I\bigr)\;\;\;\mbox{ if }\; \ga_J S = \de_J^{-1}\ga_K|_J S.
\end{align*}

  As preparation for the general case, let us check that  $\bB_S$ has a groupoid completion $\bG_S$.
If $S = \id$ then this is straightforward.   The category  $\bB_{\id}$ has objects $\bigsqcup_{I\in \Ii_Y} \Ga_I$ 
and at most one morphism  between any two points. Because  the groupoid  completion  of $\bB_{\id}$
(if it exists) has the same stabilizer subgroups as $\bB_{\id}$, the category  $\bG_{\id}$ must have 
 a single morphism between any pair of points with the same image in $Y$, and hence between each pair of objects. 
 But it is easy to construct such a groupoid.  We take
$$
\Mor_{\bG_{\id}} = {\textstyle  \bigsqcup_{I,J\in \Ii_Y}} \Ga_I\times \Ga_J, \quad s\times t(\ga_I,\ga_J) = (\ga_I,\ga_J)\in 
\Obj_{\bG_{\id}}\times \Obj_{\bG_{\id}},
$$
with  composition given by
\begin{equation}\label{eq:mult}
m_\id: (\Ga_I\times \Ga_J)\times_{\Ga_J} (\Ga_J\times \Ga_K)  \to \Ga_I\times \Ga_K,  \quad \bigl((x,y),(y,z)\bigr)\mapsto (x,z).
\end{equation}

More generally, the group $S$  acts on $\bG_{\id}$ by multiplication on the right; i.e.  each $s\in S$  gives a functor 
$F_s: \bG_{\id}\to \bG_{\id}$ that acts on objects  by $\ga\mapsto \ga s$ inducing isomorphisms
$\Mor(x,y)\to \Mor(xs,ys)$.
Since $F_s\circ F_t = F_{ts}$ for $s,t\in S$ there is a well defined quotient category $\qu{\bG_{\id}}{S}$ with objects $\bigsqcup_I \qu{\Ga_I}{S}$ and morphisms 
$\bigsqcup_{I,J} \qu{\Ga_I\times \Ga_J}{S}$.
We claim that this quotient category $\qu{\bG_{\id}}{S}$ can be identified with the groupoid completion $\bG_S$ of 
 $\bB_S$.  
 
 To prove this, consider the
 functor
$F_S: \bB_{\id}\to \bB_S$ 
 given on objects by the quotient maps $\Ga_I\mapsto \qu{\Ga_I}{S}=:W_I$,
and on morphisms (which are only defined when $I\subset J$) by
\begin{align}\label{eq:Fid}
& F_S:\Mor_{\bB_{\id}}(\Ga_I,\Ga_J)\to \Mor_{\bB_{S}}(W_I,W_J),\quad  \Ga_I\times \Ga_J \to \qu{\Ga_J}{S}\times \Ga_{I}\\\notag
& \qquad \qquad\quad 
(\ga_I,\ga_J) \mapsto \bigl(\ga_JS,(\ga_J|_I\, (\ga_I)^{-1})\bigr).
\end{align}
Then $F_S$ commutes with the target map, and commutes  with the source map because
$F_S\circ  s(\ga_I,\ga_J) = \ga_IS$ while
\begin{align*}
s\circ F_S(\ga_I,\ga_J) & = s\bigl(\ga_JS, \ga_J|_I\ (\ga_I)^{-1}\bigr) \\
& = \bigl(\ga_J|_I(\ga_I)^{-1}\bigr)^{-1} \ga_J|_IS = 
\ga_I S.  
\end{align*}
Further, $m_S\circ (F_S\times F_S) = F_S\circ m_\id$ because when $I\subset J\subset K$
\begin{align*}
m_S\circ (F_S\times F_S)\bigl((\ga_I,\ga_J),(\ga_J,\ga_K)\bigr) &= m_S\bigr((\ga_JS, \ga_J|_I\,\ga_I^{-1}), (\ga_KS, \ga_K|_J\, \ga_J^{-1})\bigr)\\
 &=(\ga_KS, \ga_K|_I\,\ga_I^{-1}) = F_S(\ga_I,\ga_K).
\end{align*}
Finally notice that $F_S\circ F_s = F_S$ for all $s\in S$
because $\ga_J^is(\ga_I^is)^{-1}=\ga_J^i(\ga_I^i)^{-1}\in \Ga_I$ when $i\in I, s\in S$.  Therefore $F_S$ descends to the quotient 
$\qu{\bB_{\id}}{S}$ (considered as a submonoid of $\qu{\bG_{\id}}{S}$), inducing an isomorphism 
from this quotient $\qu{\bB_{\id}}{S}$ 
to $\bB_S$.  We therefore  obtain from
its inverse an inclusion $\bB_S\to \qu{\bG_{\id}}{S}$  that exhibits $\qu{\bG_{\id}}{S}$ as the groupoid completion of $\bB_S$.
$\hfill\er$ \end{example}

\begin{prop}\label{prop:groupcomp}  Let $\Kk = \bigl(W_I,\Ga_I, \rho_{IJ}\bigr)_{I\subset J, I,J\in \Ii_\Kk}$ be an orbifold atlas as above. Then the category $\bB_\Kk$ 
has a canonical completion to an ep groupoid $\bG_\Kk$ with the same objects and realization as $\bB_\Kk$ and morphisms
$$
\Mor_{\bG_\Kk} = \bigsqcup_{I,J\in \Ii_Y: I\cup J\in \Ii_Y} \Mor_{\bG_\Kk}(W_I, W_J), \quad \Mor_{\bG_\Kk}(W_I, W_J): = W_{I\cup J}\times \Ga_{I\cap J},
$$
where $\Ga_\emptyset: = \id$, and with the following structural maps.
\begin{itemize}\item[\rm (i)]  The source and target maps $s\times t:  W_{I\cup J}\times \Ga_{I\cap J}\to W_I\times W_J$ are
$$
(s\times t) \bigl(z,\ga \bigr) \;=\;  \Bigl( \bigl(I, \ga^{-1}\rho_{I(I\cup J)}(z)\bigr) \,,\, \bigl( J, \rho_{J(I\cup J)}(z) \bigr) \Bigr).
$$
\item[\rm (ii)] Composition is given by
\begin{align*}
&\ m:  \Mor_{\bG_\Kk}(W_I, W_J) \, _t\times_s \Mor_{\bG_\Kk}(W_J, W_K)\to \Mor_{\bG_\Kk}(W_I, W_K), \\
&\ \bigl((z,\ga), (w,\de)\bigr) \mapsto (v',\al \,\de_{IJK}\ga_{IJK})\in W_{I\cup K}\times \Ga_{I\cap K},\quad
v': = \rho_{I\cup K,I\cup J\cup K}(v),
\end{align*}
where  $\ga_{IJK}, \de_{IJK}$ are the images of $\ga\in \Ga_{I\cap J},\  \de\in \Ga_{J\cap K}$ under projection to $\Ga_{I\cap J\cap K}$
and 
$(v,\al)\in W_{I\cup J\cup K}\times \Ga_{(I\cap K)\less J}$ is the unique pair  such that 
\vspace{-.05in}
$$
\rho_{I\cup J, I\cup J\cup K}(v) = \ga_{IJ\less K}^{-1}\al\,\de z, \;\; \rho_{J\cup K, I\cup J\cup K}(v) =  \ga_{IJ\less K}^{-1} w,
$$
where $\ga_{IJ\less K}: = \ga \ga_{IJK}^{-1}\in \Ga_{(I\cap J)\less K}$.
\vspace{.05in}
\item[\rm (iii)] The inverse is given by
$$
\io: \Mor_{\bG_\Kk}(W_I, W_J) \to \Mor_{\bG_\Kk}(W_J, W_I), \quad  \bigl(z,\ga \bigr)\mapsto (\ga^{-1} z, \ga^{-1}).
$$
\end{itemize}
\end{prop}
\begin{proof}  When $I\subset J\subset K$ the above  formulas  for $\Mor_{\bG_\Kk}(W_I,W_J)$ and the composition  in (ii) agree with the previous definitions
for $\bB_\Kk$.     We must extend this definition to all pairs 
 $I,J$ with $F_I\cap F_J\ne \emptyset$ (or equivalently $I\cup J\in \Ii_Y$) 
 so as to be consistent with the footprint maps and the local group actions.  In particular, in order to see that the inclusion $\bB_\Kk\to \bG_\Kk$ induces a homeomorphism $|\bB_\Kk|\stackrel{\cong}\to |\bG_\Kk|$   we require
\begin{itemize}\item[(*)] $\exists$  a morphism from $x\in W_I$ to $y\in W_J$ in $\bG$  iff $(I,x)\sim (J,y)$ in $\Obj_{\bB_\Kk}$
 iff $\psi_I(x) = \psi_J(y)\in F_{I\cup J}: = F_I\cap F_J$;
\end{itemize}
To see that the morphisms as described above  satisfy (*)
note first that $\Mor_\bG(x,y)\ne \emptyset$ implies that $x,y$ have the same image in $Y$.  
Conversely,
suppose  given 
 $x\in W_I,y\in W_J$ with  $I\not\subset J$ and such that $\psi_I(x) = \psi_J(y)$.
%
Since $\rho_{J(I\cup J)}: W_{I\cup J}\to W_{J(I\cup J)}: = \psi^{-1}(F_{I\cup J})$ is surjective 
and factors out by the free action of 
$\Ga_{I\less J}$
we may choose $z \in W_{I\cup J}$ so that $\rho_{J(I\cup J)}(z) = y$.  Then  
$\rho_{I(I\cup J)}(z)$ lies in the $\Ga_{I}$-orbit of $x$ because $\psi_{I\cup J}(z) = \psi_I(x)$, so that
by replacing $z$ by $\de z$ for some 
$\de\in \Ga_{I\less J}$ we may arrange  that $\rho_{I(I\cup J)}(z)$ lies in the $\Ga_{I\cap J}$-orbit of $x$,
where $\Ga_{I\cap J}: = \id$ if $I\cap J = \emptyset$.
Therefore there is a pair $(z,\ga)\in (W_{I\cup J}, \Ga_{I\cap J})$ with 
$\rho_{I(I\cup J)}(z) = \ga x, \rho_{J(I\cup J)}(z) =y$.  Thus, if we define $\Mor(W_I,W_J): = 
W_{I\cup J}\times  \Ga_{I\cap J}$ with source and target maps as in (i), condition (*) is satisfied.

The next step is to check that composition as given by (ii) is well defined.  
To this end, observe that for any triple $I,J,K$ with $I\cup J\cup K\in \Ii_Y$, 
the square in the commutative diagram
\begin{equation}\label{WWsq}
\xymatrix
{
W_{I\cup J\cup K}\quad  \ar@{->}[d]_{\rho_{I\cup J \bullet}} \ar@{->}[r]^{\rho_{J\cup K \bullet}} &\quad W_{J\cup K} \ar@{->}[d]_{\rho_{J\cup(I\cap K)\bullet}}  \\
W_{I\cup J}\quad  \ar@{->}[r]^{\rho_{J\cup(I\cap K)\bullet}} & \quad W_{J\cup(I\cap K)}\ar@{->}[r]^{\quad\rho_{J\bullet}} & W_J,
}
\end{equation}
expresses $W_{I\cup J\cup K} $ as a fiber product over $W_{J\cup(I\cap K)} = W_{(I\cup J)\cap (J\cup K)}$,
where to simplify notation we write $\rho_{I\bullet}: = \rho_{IJ}:W_J\to W_I$. 
(This holds because the projection $\rho_{I\bullet}$ is a principal $\Ga_{\bullet\less I}$-bundle.)
Hence any pair in  $W_{I\cup J}\times W_{J\cup K}$ has a unique lift to 
$W_{I\cup J\cup K}$ provided that its components have the same image in $W_{J\cup(I\cap K)}$.

Now consider the composite $(w,\de)\circ (z,\ga): W_I\to W_K$.
The identity $s(w,\de) = t(z,\ga)$ implies that
$\rho_{J(I\cup J)} (z) = \de^{-1}\rho_{J(J\cup K)} (w)$, so that the elements
$\ga_{IJ\less K}^{-1}\de\, z\in W_{I\cup J}$ and 
$\ga_{IJ\less K}^{-1}\, w\in W_{J\cup K}$ have the same image in $W_J$.  Hence, because 
$\rho_{J\bullet}: W_{J\cup(I\cap K)}\to W_J$ quotients out by a free action of $\Ga_{(I\cap K)\less J}$,
the properties of the above diagram imply 
 there is a unique $\al\in \Ga_{(I\cap K)\less J}$ such that
$$
\rho_{J(I\cup J)}(\al \ga_{IJ\less K}^{-1}\de\, z) = \rho_{J(J\cup K)}(\ga_{IJ\less K}^{-1}\, w),\;\;\mbox{ where } 
\ga_{IJ\less K}: = \rho^\Ga_{(I\cap J\less K)(I\cap J)}(\ga).
$$
(Notice here that $\al\in \Ga_K$ commutes with 
$\ga_{IJ\less K}\in \Ga_{(I\cap J)\less K}$ though it may not commute with $\de$.)
Thus $v\in W_{I\cup J\cup K}$ is uniquely defined by the given conditions.

  It remains to check that the morphism $\bigl(v',\al(\de\ga)_{IJK}\bigr)\in W_{(I\cup K)(I\cap K)}$ 
  has source $s(z,\ga)$ and target $t(w,\de)$.  
But
\begin{align*}
t(v',\al(\de\ga)_{IJK}) & = \rho_{K(I\cup K)}(v') =  \rho_{K(I\cup J\cup K)}(v) \\
& =   \rho_{K(J\cup K)}(\ga^{-1}_{I\cup J\less K}w) = \rho_{K(J\cup K)}(w) = t(w,\de),
\end{align*}
because $\ga_{IJ\less K}\in \Ga_{J\less K}$ has no component in $\Ga_K$. 
Similarly,
\begin{align*}
s(v',\al(\de\ga)_{IJK}) & = \rho_{I(I\cup K)}((\al(\de\ga)_{IJK})^{-1} v') \\
& =  (\al(\de\ga)_{IJK})^{-1} \rho_{I(I\cup J\cup K)}(v) \\
& =(\de\ga)_{IJK})^{-1}\al^{-1}\rho_{I(I\cup J)} ( \al \ga^{-1}_{IJ\less K} \de z)\\
& = (\de\ga)_{IJK})^{-1}\al^{-1} \al \ga^{-1}_{IJ\less K} \de_{IJK} \rho_{I(I\cup J)} (z)\\
& =  \ga_{IJK}^{-1} \,\ga^{-1}_{IJ\less K}\, \de_{IJK}^{-1}\, \de_{IJK}\, \rho_{I(I\cup J)} (z) = \ga^{-1}\rho_{I(I\cup J)} (z) = s(z,\ga),
\end{align*}
where we have used the fact that  $\rho_{I(I\cup J)} (\de z) = \de_{IJK}\rho_{I(I\cup J)} (z)$
because $\de\in \Ga_{J\cap K}$ so that $\de \de_{IJK}^{-1}\in \ker \rho^\Ga_{I(I\cup J)}$,
 and the fact that
  $ \ga_{IJ\less K}\in \Ga_{I\less K} $ and $\de_{IJK}\in \Ga_K$ commute.

  Thus the formula for $m$ in (ii) is well defined and compatible with source and target maps.  Moreover, $m$ is 
  a local diffeomorphism. It is also easy to check that $m$ is compatible with
  the  formula for the inverse given in (iii). Indeed, if $(z,\ga)\in \Mor(W_I,W_J)$
  it is immediate that 
  $s\times t(z,\ga) = 
 t\times s(\ga^{-1} z,\ga^{-1})=  t\times s\bigl(\io(z,\ga)\bigr)$.  Further, 
$$
m\bigl((z,\ga),(\ga^{-1} z,\ga^{-1})\bigr) = (v', \de'\ga') 
$$ 
where  $\ga' = \ga^{-1}$, $\de' = \de = \ga$, and $v' = \rho_{I(I\cup J)}(v)$ with  $v\in W_{I\cup J}$  defined by the requirement
that
$\rho_{(I\cup J)(I\cup J)}(v) = \ga^{-1} z$ so that $v= \ga^{-1} z$  and $v' = s(z,\ga)$.  Thus 
the composite $m\bigl((z,\ga),(\ga^{-1} z,\ga^{-1})\bigr) =  \bigl(s(z,\ga), \id\bigr) $ is the identity morphism at $s(z,\ga)$.  
  
To prove that $m$ is associative,\footnote
{
An alternative argument, valid in the case where the group actions are effective, is given in \cite{Mcn}.}
it suffices to show that,  for each $y\in Y$, $m$ restricts to an associative multiplication  on the 
full subcategory $\bG_y$ of $\bG_\Kk$ with objects $\bigsqcup_{I\in \Ii_y} \psi_I^{-1}(y) $, where 
$\Ii_y: = \{I\in \Ii_Y\ | \ y\in F_I\}$.
Since  $\Ga_I$ acts transitively on $\psi_I^{-1}(y)$ for each $I$, this is  precisely the case considered in Example~\ref{ex:grouppt}. 
Hence it will suffice to show that the above composition operation is the image of composition in the category $\bG_{id}$.
To this end,  choose a compatible set of base points $\bigl(x_{\bullet} = x_I\in \psi_I^{-1}(y)\bigr)_I\in \Ii_y $ and define
$$
H_I:  \Ga_I \to  \psi_I^{-1}(y),\quad \ga_I \mapsto \ga_I(x_\bullet),
$$
$$
H_{IJ}: \Mor_{\bG_{\id}}(\Ga_I,\Ga_J) \to \Mor_{\bG_y}( \psi_I^{-1}(y), \psi_J^{-1}(y)),\quad
(\ga_I,\ga_J)\mapsto (\ga'_{I\cup J}x_{\bullet}, \de'_{I\cap J})
$$
where $\de'_{I\cap J} (\ga_I|_{I\cap J}) = \ga_J|_{I\cap J}$ and $\ga'_{I\cup J}\in \Ga_{I\cup J}$ is the unique element whose projection
$\ga'_{I\cup J}|_I$ to $\Ga_I$ is $\de_{I\cap J} \ga_I$ and whose projection
$\ga'_{I\cup J}|_J$ to $\Ga_J$  is $\ga_J$.  The restriction to $\bG_y$ of the diagram \eqref{WWsq} that is used to define the composite $m$ takes the form
$$
\xymatrix
{
\Ga_{I\cup J\cup K}(x_{\bullet})\quad  \ar@{->}[d]_{\Ga_{K\less (I\cup J)}} \ar@{->}[r]^{\Ga_{I\less (J\cup K)}} &\quad \Ga_{J\cup K} (x_{\bullet})\ar@{->}[d]_{\Ga_{K\less (I\cap K)}}  \\
\Ga_{I\cup J}(x_{\bullet})\quad  \ar@{->}[r]^{\Ga_{I\less (I\cap K)}} & \quad \Ga_{J\cup(I\cap K)}(x_{\bullet})\ar@{->}
[r]^{\quad\Ga_{(I\cap K)\less J}} & \Ga_J(x_{\bullet}),
}
$$
where each arrow is labelled by the group that acts freely on its fibers.  But this is the image under $H_{\bullet}$ of a corresponding diagram for the groups $\Ga_{\bullet}$
that can be used in precisely the same way to define the composite in $\bG_{\id}$.  
This was previously written in the simple form $m_{\id}\bigl((\ga_I,\ga_J),(\ga_J,\ga_K)\bigr) = (\ga_I,\ga_K)$,
but, with $H(\ga_I,\ga_J) = (\ga'_{I\cup J},\de'_{I\cap J})$ and $H(\ga_J,\ga_K) = (\ga''_{J\cup K},\de''_{I\cap K})$
satisfying the identities
$$
\ga'_{I\cup J}|_I=\de'_{I\cap J} \ga_I,\;\; \ga'_{I\cup J}|_J=\ga_J,\;\;
\ga''_{J\cup K}|_J=\de''_{J\cap K} \ga_J,\;\; \ga''_{J\cup K}|_K=\ga_K,
$$
one can check that there is a unique element $\al\in \Ga_{I\cup J\cup K}$ with components given by
\begin{align*}
& \al|_I = ({ \de'_{I\cap J}}^{-1} \ga_J)|_I, \quad
\al|_J = ( {\de''_{J\cap K}}^{-1} \ga_K)|_J,  
\quad 
\al|_K = \ga_K
\end{align*}
that plays the role of the element $v$ in the definition of $m$ in (ii).   Using this, it is straightforward to check that the multiplications correspond under  $H_{\bullet}.$  It follows that
 $m$ is associative, which completes the definition of the groupoid $\bG_\Kk$.
\end{proof}

\section{Existence of atlases}

We now show that every orbifold has an orbifold atlas that is unique up to the following notion of commensurability.

\begin{defn}  Let $\Kk$, $\Kk''$ be orbifold atlases on $Y$.
We say that  $\Kk$ is a {\bf  subatlas} of $\Kk''$ if there is an injective \'etale functor  $\io:\bB_{\Kk}\to \bB_{\Kk''}$ such that 
$|\psi| = |\psi''|\circ |\io|: | \bB_{\Kk}|\to Y$. 
Two orbifold atlases $\Kk, \Kk'$ on $Y$ are {\bf directly commensurate} if they are subatlases  of a common atlas $\Kk''$. 
They are {\bf commensurate} there is a sequence of atlases
$\Kk =:\Kk_{1},\dots,\Kk_{\ell}: = \Kk'$ such that any consecutive pair $\Kk_i, \Kk_{i+1}$ are directly commensurate. 
\end{defn}

\begin{rmk}\rm  (i)  It is not hard to see that commensurate atlases are cobordant, i.e. there is an orbifold atlas on the product $[0,1]\times Y$ 
that restricts to $\Kk$ on $\{0\}\times Y$ and to $ \Kk'$ on $\{1\}\times Y$.  (For precise definitions, and a proof see \cite[\S6.2]{MW2}.)  
If we assume that all atlases (or, equivalently, their groupoid completions) have compatible orientations, we may conclude that any pair of atlases that are 
oriented commensurate have the same fundamental class; cf. \cite{Mbr} and \S\ref{s:app} below.\MS

\NI  (ii)
 It is likely that  commensurate atlases are directly commensurate, i.e. that the relation of direct commensurability is transitive; however we have not attempted to prove this since we are most interested in the cobordism relation.
\hfill$\er$
\end{rmk}

Here is our main result.

\begin{prop}\label{prop:orb}  Every paracompact orbifold $Y$ has an orbifold atlas $\Kk$  whose associated groupoid 
$\bG_{\Kk}$ is an orbifold structure on $Y$.   Moreover, there is a bijective correspondence between commensurability classes of such  atlases and Morita equivalence classes of ep groupoids.
\end{prop}
\begin{proof} Let $\bG$ be an ep groupoid with footprint map $f: \Obj_\bG\to Y$.
Our first aim is to construct an atlas $\Kk$ on $Y$ 
together with a functor $F: \bB_\Kk\to \bG$ that covers the identity map on  $Y$ and hence extends to
an equivalence from the groupoid completion
 $\bG_\Kk$  to $\bG$.

 By Moerdijk~\cite{Moe}, each point in $Y$ is the image of a group quotient that embeds into $\bG$.
 Therefore since $Y$ is paracompact we can find 
 a locally finite set of basic charts $\bK_i: = 
\bigl(W_i,\Ga_i,\psi_i\bigr)_{i\in A}$ on $Y$ whose footprints 
$(F_i)_{i\in A}$ cover $Y$,
together with smooth maps
$$
\si:{\textstyle \bigsqcup_iW_i\hookrightarrow \Obj_{\bG},\quad  \Tsi:\bigsqcup_i W_i\times \Ga_i\hookrightarrow \Mor_\bG},
$$
where  $\si|_{W_i}$ is a diffeomorphism to its image,
  that are compatible in the sense that the following diagrams commute:
$$
\xymatrix
{
W_i\times \Ga_i  \ar@{->}[d]_{s\times t} \ar@{->}[r]^{\Tsi_i} &  \Mor_{\bG}\ar@{->}[d]_{s\times t}    \\
W_i\times W_i \ar@{->}[r]^{\si_i} & \Obj_{\bG}\times \Obj_{\bG},
} \qquad\qquad 
\xymatrix
{
W_i  \ar@{->}[d]_{\psi_i} \ar@{->}[r]^{\si_i} &\;\;\Obj_{\bG}  \ar@{->}[d]_{f}   \\
Y \ar@{->}[r]^{\id} & Y.
}
$$
We claim that there is an atlas $\Kk$ with these basic charts whose footprint maps $\psi_I$ extend 
 $f\circ \si: \bigsqcup_i W_i \to Y$. This atlas depends on the choice of a total order on $A$.
To begin the construction, we define  $W_I$ where $|I| = 2$.  Since $A$ is ordered,
 any set $I\in \Ii_Y$ with $|I|=2$ may be written as
$I: = \{i_0, i_1\}$ with  $i_0<i_1$.  Consider
the set $$
W_I: = W_{\{i_0,i_1\}}: = \Mor_{\bG}(\si(W_{i_0}),\si(W_{i_1})): = (s_\bG\times t_\bG)^{-1}\bigl(\si(W_{i_0})\times \si(W_{i_1})\bigr)
$$
of morphisms in $\bG$ from $\si(W_{i_0})$ to $\si(W_{i_1})$, where  
 to avoid confusion the source and target maps in $\bG$ are denoted $s_\bG, t_\bG$.
Then $W_I$ is the inverse image of an open subset of $\Obj_{\bG}\times \Obj_{\bG}$, hence open in 
$ \Mor_{\bG}$, and thus a smooth manifold. 
 Since the points in $f^{-1}(F_I)\cap \si(W_{i_0})$ are identified with points in $f^{-1}(F_I)\cap \si(W_{i_1})$
 by morphisms in $\bG$,
 the restrictions of $s_\bG,t_\bG$ to $W_I$ have images
 $$
 s_\bG(W_I) = f^{-1}(F_I)\cap \si(W_{i_0}),\qquad  
  t_\bG(W_I)= f^{-1}(F_I)\cap \si(W_{i_1}).
 $$
 Moreover, for any $x\in s_\bG(W_I)$ and $\al\in \Mor_{\bG}(x,y)\in W_I$, we have
 $$
 s_\bG^{-1}(x)\cap W_I \cong \Mor_\bG\bigl(t_\bG(\al),\si(W_{i_1})\bigr)\cong \Ga_{i_1},
 $$
 where the second isomorphism holds because  by assumption
$f\circ \si_{i_1}=\psi_{i_1}: W_{i_1}\mapsto F_{i_1}$ is the composite of the 
quotient map $W_{i_1}\to \qu{W_{i_1}}{\Ga_{i_1}}$ with a homeomorphism.
Rephrasing this  in terms of the action 
of the group $\Ga_I: = \Ga_{i_1}\times \Ga_{i_0}$  on $\al\in W_{I}$ by 
 $$
(\ga_{i_1},\ga_{i_0})\cdot \al = \Tsi(\ga_{i_1})\circ \al\circ \Tsi(\ga_{i_0}^{-1}),
$$ 
one finds that  $\Ga_{i_1}$ acts freely on $W_I$ and that
the  source map $s_\bG:  W_I\to  \si(W_{i_0})$
induces a diffeomorphism $\qu{W_I}{\Ga_{i_1}}\to  \si(W_{i_0})\cap f^{-1}(F_I)$.
Similarly, $\Ga_{i_0}$ acts freely, and the target map
$t_\bG:  W_I\to  \si(W_{i_1})$
induces a diffeomorphism $\qu{W_I}{\Ga_{i_0}}\to  \si(W_{i_1})\cap f^{-1}(F_I)$.  Since the footprint map
for the chart  $W_{i}$ factors out by the  action of $\Ga_i$, the same is true for this transition chart:
in other words the footprint map
$$
\psi_I: W_I\to Y, \quad  \al\mapsto f\bigl(s_\bG(\al)\bigr) = f\bigl( t_\bG(\al)\bigr)
$$ 
 induces a homeomorphism $\qu{W_I}{\Ga_I}
\stackrel{\cong}\to F_I.
$
 Therefore $W_I$ satisfies all the requirements of a sum 
of two charts.

  To define the transition  chart for general $I \in \Ii_Y$,  enumerate the elements  of $I$ as
$i_0<i_1<\cdots<i_k$, where $k+1: = |I|\ge 2$ and define
$W_I$ to be the set of composable $k$-tuples of morphisms $(\al_{i_k},\cdots,  \al_{i_1})$, where 
\begin{equation}\label{eq:WI}
 \al_{i_\ell} \in \Mor_{\bG}\bigl(\si(W_{i_{\ell-1}}),\si(W_{i_\ell})\bigr).
\end{equation}
If $H: =  (i_1,\cdots,i_{k})$, then $W_I$ is the fiber product $W_H\ _{s_\bG}\!\times_{t_\bG} W_{i_{1} i_{0}}$.
Since the maps $s_\bG: W_{i_{1} i_{0}}\to W_{i_{0}}$, $t_\bG: W_{i_{1} i_{0}}\to W_{i_{1}}$ are \'etale and so locally submersive, it follows  by induction on $|I|$ that $W_I$ is a smooth manifold.   Moreover, it
supports an action of $\Ga_I$ given by $$
\ga\cdot (\al_{i_k},\cdots,  \al_{i_1}) = (\al_{i_k},\cdots,  \al_{i_{\ell+1}}\Tsi( \ga)^{-1}, \Tsi(\ga) \al_{i_{\ell}},\cdots,  \al_{i_1}), \quad \ga\in \Ga _{i_\ell}.
$$
For any $H\subsetneq I$ the subgroup $\Ga_{I\less H}$ acts freely, and the quotient  can be identified with $W_H$ by means of the appropriate partial compositions and forgetful maps. More precisely, if
 $I = (i_0,\cdots, i_k)\supset H = (i_{n_0},\cdots,i_{n_\ell})$  then
$$
\rho_{HI}(\al_{i_k},\cdots,  \al_{i_1}) = \left\{\begin{array}{ll}(\al_{i_{n_\ell}}\circ\cdots\circ \al_{i_{n_{\ell-1}+1}},\cdots,  \al_{i_{n_2}}\circ\cdots \circ\al_{i_{n_{1}+1}}),&\mbox{ if } \ell\ge 1\\
s_\bG(\al_{i_p+1} ) = t_\bG(\al_{i_p})&\mbox{ if } \ell=0, p: = n_0\end{array}
\right.
$$
For example if  $H = \{1,3,6\}\subset I= \{0,1,2,3,4,5,6,7\}$ then
$$
\rho_{HI}: (\al_7,\cdots,  \al_1) = (\al_6\circ \al_5\circ \al_4,  \al_3\circ \al_2), \quad
\rho_{\{3\}I} : (\al_7,\cdots,  \al_1) = s(\al_4) = t(\al_3).
$$
It is clear from this description that $\rho_{HJ} = \rho_{HI}\circ \rho_{IJ}$ whenever $H\subset I\subset J$.
Further the footprint map $\psi_I: W_I\to Y$ can be written as 
$$
\psi_I\bigl((\al_{i_k},\cdots,  \al_{i_1})\bigr) = f\bigl(\si\circ s_\bG(\al_{i_p})\bigr)= f\bigl(\si\circ t_\bG( \al_{i_p})\bigr), \quad \forall\ 1\le p\le k.
$$ 
This defines the atlas $\Kk$. 

We define the functor $F_\Kk: \bB_\Kk\to \bG$  on objects by
\begin{equation}\label{eq:FF}
W_I\to \Obj_{\bG}, \quad \left\{\begin{array}{ll}  x\mapsto \si(x), &\mbox{ if } I=\{i_0\}, x\in W_{i_0},\\
(\al_{i_k},\cdots,\al_{i_1}) \mapsto t_\bG(\al_{i_k})\in \si(W_{i_k})
&\mbox{ if } |I|>1.
\end{array}\right.
\end{equation}
Recall from  \eqref{eq:Bcomp1} that the morphisms in $\bB_\Kk$ are given by
$\bigsqcup_{I\subset J} W_J\times \Ga_I$ where
$$
(I,J,y,\ga):  \bigl(I,\ga^{-1}\rho_{IJ}(y)\bigr)\mapsto (J,y).
$$
If $i_k=j_\ell$ then we define $F_\Kk :W_J\times \Ga_I\to \Mor_{\bG}$ to be given by  the initial inclusion $\Tsi$.
More precisely,
we define
$$
F_\Kk\bigl((\al_{j_\ell},\cdots,\al_{j_1}),(\ga_{j_\ell},\cdots,\ga_{i_0})\bigr) = \Tsi(t(\al_{j_\ell}),\ga_{j_\ell}) \in \Mor_\bG\bigl(
\Tsi(\ga_{j_\ell}^{-1})\ t(\al_{j_\ell}),  t(\al_{j_\ell})\bigr),
$$
where $t$ denotes the target map in $\bB_\Kk$.
Similarly, if $i_k = j_p < j_{\ell}$ define 
$$
F_\Kk\bigl((\al_{j_\ell},\cdots,\al_{j_1}),(\ga_{i_k},\cdots,\ga_{i_0})\bigr) =(\al_{j_\ell}\circ \cdots\circ \al_{j_{p+1} })\in \Mor_{\bG}
\bigl(t(\al_{j_p}), t(\al_{j_\ell})\bigr).
$$
It is immediate that $F_\Kk$ is a functor that extends to an equivalence from the groupoid extension $\bG_\Kk$ of $\bB_\Kk$ to $\bG$.

This shows that every orbifold has an atlas of the required type. 
 To see that this atlas is unique up to commensurability, note first that 
any two atlases constructed in this way from the same  groupoid are directly commensurate.  More generally,
 suppose given  groupoid structures  $(\bG,f),(\bG',f')$ on $Y$ with
 common refinement  $$
 F:(\bG'',f'')\to (\bG,f), \quad F':(\bG'',f'')\to (\bG',f'),
 $$
  where
$F:\Obj_{\bG''}\to \Obj_\bG$ and $F':\Obj_{\bG''}\to \Obj_\bG$ are local diffeomorphisms.
Choose an atlas $\Kk''$  on $\bG''$ with basic charts $\bigl((W_i'',\Ga_i)_{i\in A})$ where for each $i$  the group $\Ga_i$ is the stabilizer subgroup of some point $x_i\in W_i''$.  Then, for each $1\le i\le N$, the map
$F:W_i''\to W_i: = F(W_i'')\subset \Obj_{\bG}$ is injective because $F$ induces an isomorphism  $\Ga_i: = \Mor_{\bG''}(x_i,x_i)\to \Mor_\bG(F(x_i),F(x_i))$ and an injection on the quotient $\qu{W''_i}{\Ga_i}\to Y$.
Therefore the basic charts $\bigl((W_i'',\Ga_i)_{i\in A})$ are pushed forward diffeomorphically by $F$  to a family of basic charts
$\bigl((W_i,\Ga_i)_{i\in A})$  in $\bG$.     Further, it is immediate from the construction of the corresponding atlases 
$\Kk''$ and $\Kk$  from the categories $\bG''$ and $\bG$ that $F$ induces an isomorphism between them.  
Hence all atlases on $Y$ that are constructed from $\bG$ or from $\bG'$ are commensurate to this atlas that is pushed forward from $\bG''$, and hence they all belong to the same commensurability class.

Conversely, we must show that if $\Kk, \Kk'$  are commensurate,
 the groupoids $\bG_{\Kk}$ and $\bG_{\Kk'}$ are equivalent.
It suffices to consider the case when $\Kk, \Kk'$ 
are  directly commensurate.  But then they are contained in a common atlas $\Kk''$ on $Y$ that defines a groupoid $\bG_{\Kk''}$
that contains both $\bG_{\Kk}$ and $\bG_{\Kk'}$ as subgroupoids with the same realization $Y$.  Thus the inclusions 
 $\bG_{\Kk}\to \bG_{\Kk''}$ and  $\bG_{\Kk'}\to \bG_{\Kk''}$ are equivalences.  This completes the proof.
\end{proof}

\begin{rmk}\rm  The above construction for the atlas $\Kk$ depends on a choice of ordering of the basic charts $(\bK_i)_{i\in A}$.  If we change this order, for example, by interchanging the order of $1$ and $2$, then it is not hard to show that
the resulting atlas $\Kk'$ is isomorphic to $\Kk$, but not in a way compatible with the functor 
$F_\Kk:\bB_\Kk\to \bG$ defined in \eqref{eq:FF}.  Indeed, each atlas $\Kk, \Kk'$ has the same basic charts, so that $F_\Kk=F'_{\Kk'} = \si$ on each $W_i$.
Moreover, the  transition charts  $\bK_I, \bK'_I$ contain precisely the same tuples as long as $\{1,2\}\not\subset I$. 
However, $W_{12} = \Mor_\bG(W_1,W_2)$ with $F_\Kk(W_{12})\subset W_2$ while 
$W_{12}': = \Mor_\bG(W_2,W_1)$ with $F'_{\Kk'}(W_{12})\subset W_1$.
The only natural map $S_{12}: W_{12} \to W_{12}'$ takes  the morphism $\al\in  \Mor_\bG(W_1,W_2)$ to $\al^{-1} \in \Mor_{\bG}(W_2,W_1)
\subset \Mor_{\bG_\Kk}$. In fact for any $I$ we may define a map
$S: \Obj_{\Bb_\Kk}\to \Obj_{\Bb'_\Kk}$ by setting
\begin{itemlist}\item
$
S_I = \id: W_{I} \to W_{I}' $ if $\{1,2\}\not\subset I$;
\item $S_I: W_I\to W_I',\;\; (\al_{i_k}, \dots,\al_{i_3},\al_2,\al_1) \mapsto  (\al_{i_k}, \dots,\al_{i_3},\al_2\circ \al_1,  \al_1^{-1})$ if 
$I=(1,2,i_3,\dots)$, where $\al_1\in \Mor(W_1,W_2), \al_2\in \Mor(W_2,W_{i_3})$ as in \eqref{eq:WI}
\end{itemlist}
We leave it to the interested reader to check that this is $\Ga_I$-equivariant, and that it extends to  a functorial isomorphism
of the groupoid completions $\bG_{\Kk} \stackrel{\cong}\to\bG_{\Kk'}$.
\hfill$\er$
\end{rmk}

\begin{rmk}\rm
The construction in Proposition~\ref{prop:orb} is reminiscent of 
that given in  \cite[\S4]{Mbr} for
the  resolution of an orbifold.  
However, the two constructions have  different aims: here we want to build a simple model for $Y=|\bG|$,
 while there we wanted to construct a nonsingular  \lq\lq resolution", i.e.
 a corresponding weighted  branched manifold with the same fundamental class.
 We explain below how 
our current methods  simplify the construction of such a resolution.
$\hfill\er$  \end{rmk}

\begin{example}\label{ex:gerbe}\rm {\bf (Noneffective orbifold structures on $S^2$.)} 
Consider an orbifold structure $\bG$ on $Y=S^2$ that locally has the form $\R^2/\Ga$ where 
$\Ga: = \Z/2\Z$ acts trivially.  
These are classified by the topological type of the corresponding classifying space $B\bG$ (see \cite{ALR}), which is  a bundle over $S^2$ with fiber $B\Z_2 = K(\Z/2,1)$.  Hence there are two such orbifolds,  
the trivial orbifold 
which has an atlas $\Kk^{triv}$ with a single chart $(S^2,\Z_2, \psi=\id)$  and one other.  They
can be distinguished either by an 
element in $H^2(S^2,\Z_2)\cong \Z_2$ or by the fact that in the trivial case the bundle $ B\bG\to S^2$ has a section. 
One can see both  these kinds of twisting from suitable atlases.

For example, consider an atlas with two basic charts with footprints 
equal to discs $(F_i)_{i=1,2}$ that intersect in an annulus $F_{12}$.  If each has the trivial action of $\Ga_i = \Z_2$,
we may identify the domains $W_i$ with $F_i$ via the footprint maps $\psi_i$, and hence identify
the covering maps $\rho_{i,12}: W_{12}\to W_{i,12}\subset W_i$ with the footprint map
$\psi_{12}:  W_{12}\to F_{12}$.  If $\psi_{12}$ is the nontrivial $2$-fold covering of the annulus, one can easily see that the
boundary map $\pi_2(S^2)\to \pi_1(B \bG_\Kk)$ of the fibration $B \bG_\Kk\to S^2$
is nonzero, so that this  atlas describes the nontrivial orbifold.

On the other hand, suppose we choose an atlas whose footprints $F_I$ are all contractible.  
Then $W_I$ is a union of $2^{|I|-1}$ copies of $F_I$ that are permuted by the action of $\Ga_I$, with the diagonal subgroup acting trivially. For example, the basic charts have $W_i\cong F_i$,
 the charts with $|I| = 2$ have $W_I$ equal to two copies of $F_I$ that are  permuted by the actions of $\Ga_i, i\in I$,
 while the charts with with $|J|=3$ have $W_J$ equal to $4$ copies of $F_J$.
 From this information, we can build a \v{C}ech cocycle representative
 $(\al_J: F_J\to \Z_2)_{|J|=3}$ for an element of $H^2(S^2,\Z_2)$
 by choosing one component $W_I^0$ of $W_I$  for each $|I|=2$, and then defining $\al_J: = 0$ if there is a component
 $W_J^0$
 of $W_J$  such that $\rho_{IJ}(W_J^0) = W_I^0$ for all $I\subset J, |I|=2$, and setting $\al_J: = 1$ otherwise.
 Notice that this information captures the structure of the triple intersections since there are only two possibilities:
 if $J = \{j_1,j_2,j_3\}$, then, because the groups $\Ga_{J\less  j_k}$ act freely on $W_J$ for $k=1,2,3$, there is precisely one component 
 of $W_J$ that projects to $W^0_{J\less i_k}$ for $k = 1,2$ and it either does or does not map to $W^0_{J\less i_3}$.
 If we suppose in addition that all fourfold intersections are empty, then $(\al_J)$ is a cocycle.  Moreover, 
 it represents the trivial cohomology class if and only if one can choose a family of components $W_I^0$ of the domains
that are compatible with the projections $\rho_{IJ}$ and hence form the space of objects of a nonsingular\footnote
 {
 i.e. there is at most one morphism between any two objects}  subgroupoid 
$\bG^0_{\Kk}$
of $\bG_{\Kk}$ with realization $S^2$. Since the classifying space of such a subgroupoid $\bG^0_{\Kk}$ would provide a section of the bundle
$B \bG_\Kk\to S^2$, the  triviality of the cocycle implies that the atlas defines the trivial orbifold structure on $S^2$
Conversely one can check that if the groupoid $\bG_\Kk$ defines the trivial structure  then it has a nonsingular 
subgroupoid with realization $S^2$: indeed, since such subgroupoids can be pulled back and pushed forward by equivalences,
 any groupoid that is Morita equivalent to the trivial groupoid $\bG_{\Kk^{triv}}$ contains such a subatlas.  Thus the 
  cocycle described above does classify these orbifold structures.
  
In the above discussion we assumed for simplicity that all fourfold intersections are empty. However, it is not hard to check that 
$(\al_J)$ is always a cocycle so that the above argument goes through for any cover of  $S^2$.
For this, we must show that, for every $K$ with $ |K|=4$, an even number of the four terms $\al_{K\less j_k}$ are
zero. To this end, consider $K = \{1,2,3,4\}$ and suppose that $\al_{123} = 0$.  Let $W_{123}^0$ be the component of $W_{123}$ that projects to $W_{12}^0, W_{13}^0,W_{23}^0$ and let $W_{1234}^0$ be the unique component of $W_{1234}$ that projects to 
 $W_{123}^0$ and $W_{14}^0$.  If in addition it projects to  $W_{k4}^0$ for $k=2$ or $k=3$  then its image in $W_{1k4}$ projects to
 $W_{1k}^0, W_{14}^0,W_{k4}^0$ so that $\al_{1k4} = 0$.    But if  $W_{1234}^0$ projects to neither of $W_{24}^0,W_{34}^0$ then its image
 $\ga_4W_{1234}^0$ under the nontrivial element $\ga_4\in \Ga_4$ projects to  $W_{23}^0, W_{24}^0,$ and $W_{34}^0$ so that $\al_{234}=0$.
 Therefore, at least two of the $\al_{K\less j_k}$ vanish.  On the other hand if three of them vanish, say $\al_{123},\al_{124}, \al_{134}$, then the component
$W_{1234}^0$ defined above must project to $W_{K\less k}^0$ for $k=2,3,4$ and hence project to all $W_{ij}^0$ for $i,j\in \{1,2,3,4\}$.  Therefore
we may take $W_{234}^0$ equal to its image in $W_{234}$; in other words $\al_{234}=0$ as well.
  $\hfill\er$
  \end{example}

\section{Applications}\label{s:app}
   We give two applications of our methods, first showing how the zero set construction in \cite{MW3} gives a simple way to 
construct a nonsingular resolution of an orbifold, and second using this to construct a weighted branched manifold that represents the Euler class of an orbibundle.
  
 We begin by defining the notion of a {\bf resolution} of an ep groupoid $\bG$.    This is obtained from a groupoid by 
 first passing to a suitable Morita equivalent groupoid by pulling back via an open cover of the objects (a process called {\bf reduction})
 and then
discarding some of its morphisms.
The idea is to obtain a \lq\lq simpler" groupoid that still has the same fundamental class; the groupoid is simpler in the sense that all stabilizers are trivial, however, because it is not proper, one must
 control its branching as explained below.
  
 First recall that the realization of an ep groupoid $\bG$  carries a 
  weighting function $\La_G:|\bG|\to \Q^+$ with values in the positive rational numbers $\Q^+$,
  given by:  $\La_G(y) = \frac 1{|\Ga^y|}$, where $|\Ga^y|$ is the order of the stabilizer subgroup $\Ga^y$ at one (and hence any) preimage of 
  $y$ in $\Obj_{\bG}$.  If $\bG$ is oriented and compact, the set of points $|\bG|^*$ where $|\Ga^y|$ is locally constant is open and dense, with complement of codimension $\ge 2$, and hence carries a fundamental class that can be represented by the singular cycle obtained by triangulating $|\bG|^*$, giving each top dimensional simplex $\si$ the weight $\La_G(y), y\in \si$. (For more details, see \cite{MPr,Mbr}.)
 
Roughly speaking, a  {\bf resolution} of an oriented compact  ep groupoid  $\bG$ is a tuple
 $(\bV,\La_V,F)$ consisting of  
 \begin{itemize}\item[-]  an oriented nonsingular \'etale groupoid $\bV$ 
 (more precisely a wnb groupoid) whose realization  carries a weighting function $\La_V:|\bV|_\Hh\to \Q^+$ where $|\bV|_\Hh$ is the maximal Hausdorff quotient of $|\bV|$ (defined below), together with
 \item[-]  an orientation preserving  functor $F: \bV\to \bG$ that induces a surjection $|F|: |\bV|_\Hh\to |\bG|$ and is such that
 $F_*(\La_V)=\La_G$, where  the pushforward   $F_*(\La_V):|\bG|\to \Q^+$ is given by
 $F_*(\La_V) = \sum_{x\in F^{-1}(y)} \La_G(y)$.
 \end{itemize}
  As in \cite{Mbr}, one can define the notion of the fundamental class of  $(\bV,\La_V,F)$, and show that under these circumstances $F$ pushes this fundamental 
  class forward to that of $\bG$. 
    
   To make the above  precise, we must define a wnb groupoid.  Because these
 are in general not  proper, the realization $|\bG|$ may not be Hausdorff, and we write 
$|\bG|_{\Hh}$ for its {\bf maximal Hausdorff quotient}. Thus  $|\bG|_\Hh$ is a Hausdorff quotient of $|\bG|$ that  satisfies the 
 following universal property: any continuous map from $|\bG|$ to a Hausdorff  space $Y$ factors though 
 the projection
 $|\bG|\to |\bG|_\Hh$.  (The existence of such a quotient for any topological space is proved in \cite[Lemma~3.1]{Mbr}; see  \cite[Appendix]{MW3} for a more detailed argument.) 
There are natural maps:
$$
\pi_\bG:\Obj_{\bG}\to |\bG|,\qquad   \pi_{|\bG|}^\Hh: |\bG|\longrightarrow |\bG|_\Hh, \qquad 
\pi^\Hh_\bG:= \pi_{|\bG|}^\Hh\circ \pi_\bG : \Obj_{\bG}\to |\bG|_\Hh .
 $$
Moreover, for $U\subset \Obj_\bG$ we write $|U| := \pi_\bG(U) \subset |\bG|$ and $|U|_\Hh := \pi_\Hh(U)\subset |\bG|_\Hh$.
The {\bf branch locus} of $\bG$ is defined to be the subset of $|\bG|_\Hh$ consisting of points with more than one preimage in $|\bG|$.

 \begin{defn} \label{def:wnb}
A   {\bf weighted nonsingular branched groupoid} (or {\bf wnb groupoid} for short) of dimension $d$
is a pair $(\bG,\La_G)$ consisting of an oriented, nonsingular \'etale  groupoid 
$\bG$ of dimension $d$, together with a rational weighting function $\La_G:|\bG|_{\Hh}\to \Q^+: = \Q\cap (0,\infty)$ that satisfies the following compatibility conditions.
For each $p\in |\bG|_{\Hh}$ there is an open neighbourhood $N\subset|\bG|_{\Hh}$ of $p$,
a collection  $U_1,\dots,U_\ell$ of disjoint open subsets of $(\pi_{\bG}^{\Hh})^{-1}(N)\subset \Obj_{\bG}$ (called {\bf local branches}), and a set of positive rational weights $m_1,\dots,m_\ell$ such that the following holds: 
\SSS

\NI
{\bf (Covering) } $( \pi_{|\bG|}^{\Hh})^{-1}(N) = |U_1|\cup\dots \cup |U_\ell| \subset |\bG|$;
\SSS

\NI
{\bf (Local Regularity)}  
for each $i=1,\dots,\ell$ the projection 
$\pi_{\bG}^{\Hh}|_{U_i}: U_i\to |\bG|_{\Hh}$ is a homeomorphism onto a relatively closed subset of $N$;
\SSS

\NI
{\bf (Weighting)}  
for all $q\in N$, 
the number 
$\La_G(q)$ is the sum of the weights of the local
branches whose image contains $q$:
$$
\La_G(q) = 
\underset{i:q\in |U_i|_{\Hh}}\sum m_i.
$$
\end{defn}

 Further we define a 
{\bf weighted branched manifold} of dimension $d$ to be a pair $(Z, \La_Z)$ consisting of a topological space $Z$ together 
with a function $\La_Z:Z\to \Q^+$ 
and an equivalence class\footnote{
The precise notion of equivalence is given in \cite[Definition~3.12]{Mbr}. In particular it ensures that the induced function $\La_Z: = \La_\bG\circ f^{-1}$  the dimension of $\Obj_{\bG}$ and the pushforward of the fundamental class are the same for equivalent structures $(\bG,\La_G, f)$. However, it does not preserve the local branching structure of $Z$.  
} 
of wnb $d$-dimensional groupoids $(\bG,\La_G)$ and homeomorphisms $f:|\bG|_\Hh\to Z$ that induce the function $\La_Z = \La_G\circ f^{-1}$.
Analogous definitions of a wnb cobordism groupoid (always assumed to be compact and  have collared boundaries) and of a weighted branched cobordism 
are spelled out in \cite[Appendix]{MW3}.    We say that two compact weighted branched manifolds $(\p^\al Z,\La^\al)_{\al=0,1}$
are {\bf cobordant} if they form the oriented boundary of a  weighted branched cobordism.

\begin{example}\rm  (i)  A compact weighted branched manifold
of dimension $0$  consists of
a finite set of points $Z=\{p_1,\ldots,p_k\}$, each with a positive rational weight $m(p_i)\in \Q^+$  and orientation $\mathfrak o(p_i)\in \{\pm\}$. 
\MS

\NI (ii) The prototypical example of a $1$-dimensional weighted branched cobordism $(|\bG|_\Hh,\La)$
has  $\Obj(\bG)=I\sqcup I'$ equal to two copies of the interval $I=I'=[0,1]$ with nonidentity morphisms from $x\in I$ to $x\in I'$ for $x\in [0,\frac 12)$ and their inverses,
 where we suppose that $I$ is oriented in the standard way.
 Then the realization and its Hausdorff quotient are
$$
\begin{array}{cl}
 |\bG| &= \;\qq{I\sqcup I'}{\bigl\{(I,x)\sim (I',x) \; \text{iff}\;  x\in [0,\tfrac 12)\bigr\}},\\
 |\bG|_\Hh& = \;\qq{I\sqcup I'}{\bigl\{(I,x)\sim (I',x) \; \text{iff}\;  x\in [0,\tfrac 12]\bigr\}},
\end{array}
$$
and the branch locus is a single point $\Br(\bG)=
\bigl\{[I,\frac 12]=[I',\frac 12]\bigr\}
\subset |\bG|_\Hh$.
The choice of weights $m, m'>0$ on 
the two local branches
$I$ and $I'$ determines the weighting function $\La: |\bG|_\Hh \to (0,\infty)$ as
$$
\La([I,x])
 = \left\{
 \begin{array}{ll} 
 m+m'  & \mbox{ if }\;\; x\in [0,\frac 12],  \vspace{.1in}\\
m &\mbox{ if }\;\;  x\in (\frac 12,1],
\end{array}\right. 
\qquad
\La([I', x]) 
 = \left\{
 \begin{array}{ll} 
m+m'  & \mbox{ if }\;\; x\in [0,\frac 12],  \vspace{.1in}\\
m'  &\mbox{ if }\;\;  x\in (\frac 12,1].   
\end{array}\right. 
$$
\NI (iii)  It is not hard to see that
a wnb groupoid $\bZ: = \bigl( (p_i), m, {\mathfrak o}\bigr)$ of dimension $0$ is cobordant either to the empty  groupoid
 (if $\la: =\sum_i {\mathfrak o}(p_i) m(p_i) = 0$) or to a groupoid with one point $p$,  weight $m(p): = |\la|$ and orientation
 ${\mathfrak o}(p)$ given by the sign of $\la$.
Indeed suppose that $$
\la^+: = 
\sum_{i:
{\mathfrak o}(p_i) =+} m(p_i)\; >\; \la^-: = \sum_{i:
{\mathfrak o}(p_i) =-} m(p_i).
$$
Then 
one can first build a cobordism as in (ii) from $\bZ$ to a groupoid with two points, $p^+$ with label $(\la^+,+)$ and $p^-$ with label $(\la^-,-)$, then split $p^+$ into two labelled points  $(q_1,\la^+-\la^-,+), (q_2, \la^-,+)$ and then \lq\lq cancel"
$(q_2, \la^-,+)$ with $(p_-, \la^-,-)$ by joining them with an arc.  The other cases are similar. 
Thus in dimension $0$  the only cobordism  invariant  of a wnb groupoid is the total weight $\sum {\mathfrak o}(p_i) m(p_i)$.
$ \hfill\er$
\end{example}

Before constructing the resolution we need one further notion.  We restrict to the compact case for simplicity.

\begin{defn}\label{def:resol}
Let $(F_i)_{i=1,\dots,N}$ be an open covering of a space $Y$, and for $I\subset \{1,\dots,N\}$ denote $F_I: = \bigcap_{i\in I} F_i$.
A collection of open sets $(Q_I)_{I\subset \{1,\dots,N\}}$ is called a {\bf cover reduction} of $(F_i)$ if
\begin{itemize}\item $Q_I$ is a precompact subset of $F_I$ for all $I$, written $Q_I\sqsubset F_I$;
\item $\bigcup_I Q_I = Y$;
\item $\ov Q_I\cap \ov Q_J\ne \emptyset \Longrightarrow \bigl(I\subset J \mbox { or } I\subset J\bigr)$.
\end{itemize}
\end{defn}

It is well known that every finite open cover of a normal topological  space has a cover reduction: see for example  
\cite[Lemma~5.3.1]{MW1} for a proof.
\MS

  Let $\Kk$ be a strict  orbifold atlas on a compact oriented orbifold $Y$ with footprint covering $(F_i)_{i=1,\dots,N}$ and charts indexed by $\Ii_Y$, 
and let $\bB_\Kk$ be the  corresponding category with groupoid completion $\bG_\Kk$.  Choose a reduction $(Q_I)_{I\in \Ii_Y}$ of the footprint cover, and define\footnote
{
We write $\TV_{IJ}$ here to emphasize that, in distinction to the set $W_{IJ} = W_I\cap \psi_I^{-1}(F_J) \subset W_I$,
we have $\TV_{IJ}\subset V_J$.  This notation is  consistent with \cite{MW3,Mcn}. Note also that $\TV_{JJ} = V_J$. 
}
 $$
V_I: = \psi_I^{-1}(Q_I)\sqsubset W_I,\quad \TV_{IJ}: = V_J\cap \psi_J^{-1}(Q_I)\sqsubset W_J,\;\; \forall I\subset J.
$$
\begin{defn} The resulting collection of sets $\Vv: = (V_I)_{I\in \Ii_Y}$ is called a {\bf reduction} of the atlas.
\end{defn}

Given a reduction $\Vv$,  consider the 
 subgroupoid $\bV_\Kk\subset \bG_\Kk$ with
\begin{equation}\label{eq:Vv}
\Obj_{\bV_\Kk}: = \bigsqcup_{I\in \Ii_Y} V_I, \qquad \Mor_{\bV_\Kk}: = \bigsqcup_{I,J\in \Ii_Y} \Mor_{\bV_\Kk}(V_I,V_J),
\end{equation}
\vspace{-.05in} 
 where
\begin{itemize} \item if $I\subset J$ then
$\Mor_{\bV_\Kk}(V_I,V_J)=\bigcup_{\emptyset\ne K\subset I} ( \TV_{KJ}\cap  \TV_{IJ}) \times \Ga_{I\less K}\subset \Mor_{\bG_\Kk}(V_I,V_J)$ 
\item if $I\supset J$ then $\Mor_{\bV_\Kk}(V_I,V_J) = \{\mu^{-1}: \mu\in \Mor_{\bV_\Kk}(V_J,V_I)\}$.
\item  $\Mor_{\bV_\Kk}(V_I,V_J) = \emptyset $ otherwise.
\end{itemize} 
Note that $\bV_\Kk$ is {\it not} a full subcategory of $\bG_\Kk$: for example, we do not include all the morphisms $V_J\times \Ga_J$ from $V_J$ to $V_J$ but (besides the identities) just those with source (and hence target) in one of the sets $\TV_{IJ}, I\ne J,$ 
and over these points we  include only the action of the subgroup $\Ga_{J\less I}$,  which by 
definition of an atlas, is {\it free}.  This is the key reason why $\bV_\Kk$ is nonsingular.
Another way of understanding $\bV_\Kk$ is to see that its morphisms are generated by the projections $\rho_{IJ}: \TV_{IJ}\to V_I$.  
When $I\subsetneq J$,
each $x\in \rho_{IJ}(\TV_{IJ})$
has preimage $\rho_{IJ}^{-1}(x) $ consisting of the free orbit $\Ga_{J\less I}(\Tx)$ for $\Tx \in\rho_{IJ}^{-1}(x) \subset \TV_{IJ}$, and we recover the action $(J,\ga^{-1}\Tx)\mapsto (J,\Tx)$ of 
$\Ga_{J\less I}$ on $\TV_{IJ}$  as the set of composites
\begin{align*}
(J,J,  \Tx, \ga) &\ = (I,J, \ga^{-1} \Tx,\id)^{-1}\circ  (I,J,  \Tx,\id),  \\
(J,\ga^{-1}\Tx)&\ \mapsto \bigl(I,\rho_{IJ}(\ga^{-1}\Tx)\bigr)\ =\ \bigl(I, \rho_{IJ}(\Tx)\bigr)\mapsto (J,\Tx).
\end{align*}
where we use the notation in \eqref{eq:Bcomp}, and in particular categorical order for composites. 

Here is the main result about the groupoid $\bV_\Kk$ from \cite[Thm.~3.2.8]{MW3}.

\begin{prop}\label{prop:V}
For each  orbifold atlas $\Kk$ on $Y$, the following statements hold.
 \begin{enumerate}
\item The groupoid $\bV_\Kk$ is well defined, in particular its set of morphisms is closed under composition and taking the inverse.  
\item Its maximal Hausdorff quotient $|\bV_\Kk|_\Hh$ is the realization of the \'etale groupoid  $\bV_\Kk^\Hh$  obtained from $\bV_\Kk$ by closing its space of morphisms in 
$\Mor_{\bG_\Vv}$, where $\bG_\Vv$ is the full subcategory of $\bG_\Kk$ with objects $\Vv : = \bigsqcup_I V_I$.
\item $\bV_\Kk$ may be given the structure of a wnb groupoid with   weighting function given at  $y\in \pi_{\bV_\Kk}^{\Hh}(V_J)$  by
$$
\La_V(y) = \frac{ n(y)}{|\Ga_{J}|}, \qquad 
n(y) = \# \{x\in V_J \ | \ \pi_{\bV_\Kk}^{\Hh}(x) = y\}.
$$
Further, for $y\in \pi_{\bV_\Kk}^{\Hh}(V_J)$ the inverse image $V_J\cap (\pi_{\bV_\Kk}^{\Hh})^{-1}(y)$ is a free $\Ga_{J\less I_y}$-orbit, where
$I_y : = \min \{I\subset J \ | \ y\in \ov{\pi_{\bV_\Kk}^{\Hh}(V_I)}\}$.
\item The inclusion $\bV_\Kk\to \bG_\Kk|_{\Vv}$  extends to an inclusion $\io_{\bV_\Kk^\Hh}: \bV_\Kk^\Hh\to \bG_\Kk|_{\Vv}$.
Moreover the pushforward of $\La_V$  by  $|\psi|\circ |\io_{\bV_\Kk^\Hh}|: |\bV_\Kk|_\Hh= |\bV_\Kk^\Hh| \to  Y$ is $\La_Y$.
\end{enumerate}
\end{prop}
\begin{proof}  We  sketch the proof very briefly; \cite[\S3.4]{Mcn} gives more detail, while the full proof is in \cite[\S3.2]{MW3}.
The first claim is not hard to prove from the  remarks after the definition of $\bV_\Kk$.  To prove (ii) it suffices to check that the closure of $\Mor_{\bV_\Kk}$ in 
$\Mor_{\bG_\Kk}$ defines a set of morphisms that is closed under composition.  This holds for much the same reason as (i) because, as is easily seen,
one can close $\Mor_{\bV_\Kk}$ by adding in morphisms of the following type from $V_I$ to $V_J$:
$$
\bigcup_{F\subsetneq I}  \bigl(\TV_{IJ}\cap \Fr_{V_J}(\TV_{FJ})\bigr) \times \Ga_{I\less F}\;\subset \; V_J\times \Ga_I,
$$
where $\Fr_V(A):  = cl_V(A) \less A$ and $\cl_V(A)$ is given by the closure of $A$ in $V$.
Informally one can think of the sets $V_J$ as the branches of $\bV_\Kk$ each weighted by $\frac 1{|\Ga_J|}$.  However, they do not inject into $|\bV_\Kk|$ (and hence into $|\bV_\Kk|_\Hh$) --- rather they 
are wrapped around themselves by partial actions of the groups $\Ga_{J\less I}$.  One can check that 
 the branch locus is the image in $|\bV_\Kk|_\Hh$ of the sets $\Fr_{V_J}(\TV_{IJ})$ for $I\subsetneq J$.  
 The statements in (iii) then follow easily.
Note that although the functor
 $\io_{\bV_\Kk^\Hh}: \bV_\Kk^\Hh\to \bG_\Kk|_{\Vv}$ is injective, its image is not usually a full subcategory,  so that the induced map on realizations is not injective in general.
\end{proof}

\begin{example}\label{ex:foot3}\rm (i)  Consider the \lq\lq football" discussed in Example~\ref{ex:foot}, with  reduction $\Vv$  given by two discs
 $V_1\sqsubset W_1, V_2\sqsubset W_2$ with disjoint images $Q_i$ in $X$, together with an open annulus $V_{12}\sqsubset W_{12}$.  
 For $j = 1,2$ the sets $\TV_{j(12)}\subset V_{12}$ are disjoint open annuli that project  into
$V_j$ by a covering map of degree $3$ for $j=1$ (that quotients out by $\Ga_{(12)\less 1} = \Ga_2 = \Z_3$) and degree $2$ for $j=2$.  
Then $\Obj_{\bV_\Kk} = V_1\sqcup V_2\sqcup V_{12}$.  For $j = 1,2$ the category $\bV_\Kk$ 
has the following morphisms (besides identities);
\begin{itemize}\item morphisms  $V_j\to V_{12}$ given by the projection $\rho_{j, 12}: \TV_{j(12)} \to V_j$, together with their inverses;
\item morphisms $V_{12}\to V_{12}$ given by the action of $\Z_3 = \Ga_{(12)\less 1}$ on $\TV_{1(12)}$, resp. of $\Z_2
=\Ga_{(12)\less 2} $ on $\TV_{2(12)}$.
\end{itemize}
To obtain $\bV_\Kk^\Hh$ we add the morphisms given by the action of $\Ga_{(12)\less 1}$ on 
the boundary $\Fr_{V_{12}}(\TV_{1(12)})\subset  V_{12}\less \TV_{1(12)}$  and the action of 
$\Ga_{(12)\less 2}$ on $\Fr_{V_{12}}(\TV_{2(12)})\subset V_{12}\less \TV_{2(12)}$. 
The weighting function $\La$ is given by:
\begin{align*}
\La(p) & \ = \tfrac 12 \mbox { if } p\in \ov{Q_1} = \pi_{\bV}^{\Hh}(V_1) \cup \pi_{\bV}^{\Hh}(\ov{\TV_{1(12)}}) \\
& \ = \tfrac 13 \mbox { if } p\in \ov{Q_2}= \pi_{\bV}^{\Hh}(V_2) \cup \pi_{\bV}^{\Hh}(\ov{\TV_{2(12)}}) \\
& \ = \tfrac 16 \mbox { if } p\in Q_{12}\less (\ov {Q_1}\cup \ov{Q_2}) = \pi_{\bV}^{\Hh}(V_{12} \less \ov{\TV_{1(12)} \cup \TV_{2(12)}})
\end{align*}
Notice that for $j = 1,2$  the weighting function does not change along the boundary of the intersection 
$Q_1\cap \p \ov{Q_{12}}$,
i.e. there is no branching there, while it does change along
the internal boundaries $Q_{12}\cap \p \ov{Q_i}$  
 in the middle annulus $Q_{12}$. 
 Also, the pushforward of $\La_V$ by the map $|\io_{\bV_\Hh}|: |\bV_\Hh|\to |\bB_\Kk\big|_\Vv|$ takes the value $1$ 
except at the  poles $N,S$:
 $$
  |\io_{\bV_\Hh}|_*(\La_V)(q): = \sum_{p\in  |\io_{\bZ_\Hh}|^{-1}(q)} \ \La_\bV(p) = 1, \quad \forall q\in  Y\less \{N,S\}. 
  $$
  
  \NI (ii)  In Example~\ref{ex:gerbe} we considered the two different orbifold structures on $S^2$ with noneffective group $\Z_2$,
  constructing atlases with two basic charts
  whose footprints intersect in an annulus $F_{12}$.    They may be distinguished by
  the domain $W_{12}$, which is either  connected (the nontrivial case) or disconnected.
  Let us choose the footprint reduction so that $Q_{12}$ is a connected annulus.  Then
  because we define $V_{12}$ to be the full inverse image of $Q_{12}$ under the footprint map, it is disconnected exactly  if $W_{12} $ is.
Therefore   the two resulting weighted branched manifolds $(Z,\La_Z)$, which have  
 two-fold branching along $Q_{12} \cap \p \ov{Q_i}$
 as in (i), may be distinguished 
  by the set of points in the realization $Z=|\bV_\Kk|_\Hh$ with weight $\frac 14$: this set is either connected (the nontrivial case) or disconnected. 
  
  Observe that each of these weighted branched manifolds is weighted branched cobordant to 
  $S^2$ with the constant weight function $\frac 12$.  
  In other words, the difference between 
  these two orbifold structures  is {\it not} preserved when we consider cobordism classes of resolutions.
  To see this, notice that in each case
  we may   add morphisms to the groupoid $\bV_\Kk$ so that it still remains nonsingular but has realization $S^2$ instead of a branched manifold: to do this we simply add one morphism between any two points $(I,x), (J,y) \in \Obj_{\bV_\Kk}$
  that have the same image under the composite map $\Obj_{\bV_\Kk}\to\Obj_{ \bG_\Kk} \to |\bG_\Kk| = S^2$
  but different images in $|\bV_\Kk|$.  (Because $\bV_\Kk$ is nonsingular there is no ambiguity about how to define composites.)
  One can check that
  this new groupoid $\bV_\Kk'$ is weighted cobordant to $\bV_\Kk$ by a cobordism groupoid  $\bC$
  obtained by adding the morphisms $[0,\frac 12)\times\bigl( \Mor_{\bV_\Kk'}\less \Mor_{\bV_\Kk}\bigr)$ to the product
  groupoid
  $[0,1]\times \bV_\Kk$ (which has objects $[0,1]\times \Obj_{\bV_\Kk}$ and morphisms $[0,1]\times \Mor_{\bV_\Kk}$).
  The Hausdorff realization of this cobordism is the union of $S^2\times [0,\frac 12]$ with weighting function $\frac 12$,  together with $(\frac 12,1]\times Z$ with weighting function $\la_Z\circ pr_Z$, where as above $Z: = |\bV_\Kk|_\Hh$.  
  $   \hfill\er$
\end{example}

\begin{rmk}\rm \label{rmk:cobordinvar} Because any two choices of cover reductions are cobordant (see \cite[Lemma~5.3.4]{MW1}), one can easily show
 that if two orbifold atlases $\Kk_0,\Kk_1$ on $Y$ are commensurate 
then any two 
 resolutions  $\Vv_{\Kk_0}, \Vv_{\Kk_1}$ that are constructed as above are  themselves weighted branched cobordant.
As Example~\ref{ex:foot3}~(ii) shows, inequivalent atlases may have  cobordant resolutions.
On the other hand, the Pontryagin numbers are invariants of weighted branched cobordism. 
To see this, note that each wnb groupoid $(\bG, \La_G)$  has a tangent bundle $\rT \bG$ that is an \'etale groupoid which
(after appropriate taming) also has a natural structure as
a wnb groupoid.\footnote
{
The issue here is that the Hausdorff completion $|\rT \bG|_\Hh$  should also form a bundle over $|\bG|_\Hh$, which is the case 
when the branch locus is sufficiently well behaved.  Such questions are discussed at length in \cite[\S3]{Mbr}, where it is  shown that
\lq\lq tame" wnb groupoids support partitions of unity, and, if compact,  support a well defined notion of the  integral of a top dimensional differential form.}
Hence one can use Chern--Weil theory to construct top-dimensional differential forms that represent products of Pontryagin classes, and then integrate them over the fundamental class of  $(\bG, \La_G)$ to obtain the Pontryagin numbers. 
More generally, one could consider the bordism groups of maps from a weighted branched manifold into a space $Y$.
See \cite[Example 9.23]{CMS} for a related discussion. (The notion of  weighted branched manifold in \cite{CMS} is closely related to ours, 
but not precisely the same.)
 \hfill$\er$
\end{rmk}

\NI{\bf Computing the Euler class.}  By definition, an oriented orbibundle $\pr:E\to X$ with fiber $E_0$ over a smooth oriented compact orbifold $X$  is the realization of a smooth functor $\pr:\bE\to \bX$ between oriented ep groupoids such that the induced map $\pi_0: \Obj_{\bE}\to \Obj_{\bX}$ on objects is a locally trivial vector bundle with fiber $E_0$.
In this situation, the orbifolds $E=|\bE|, X=|\bX|$ 
%
have compatible local uniformizers.  In other words,
we may choose
 a covering of $X$ by local charts $\bigl((W_i,\Ga_i,\psi_i^X)\bigr)_{i=1,\dots,N}$ with footprints $F_i\subset X$ so that
 the action of $\Ga_i$ lifts to the pullback $(\psi^X_i)^*(E|_{F_i})$ and
    $((\psi^X_i)^*(E|_{F_i}),\Ga_i, \psi_i^E)$  (where $\psi_i^E$ lifts $\psi_i^X$) is a local uniformizer for $E$.
By Proposition~\ref{prop:orb} we may extend this family of basic charts to an orbifold atlas 
$\Kk_X $ on $X$ with  charts $\bigl((W_I,\Ga_I,\psi^X_I)\bigr)_{I\in \Ii_X}$ and
 footprint cover $(F_i)_i$.  
The orbifold $E$  has a corresponding atlas
 $\Kk_E$  with  charts $\bigl((E|_{W_I},\Ga_I,\psi^E_I)\bigr)_{I\in \Ii_X}$ and
 footprint cover $(E|_{F_i})_i\subset E$, where for simplicity  we denote the pullback 
 $(\psi^X_I)^*(E|_{F_I})$ of $E$ to $W_I$ simply by $E|_{W_I}$.
 
 By Proposition~\ref{prop:groupcomp} the categories 
 $$
 \bB_X: = \bB_{\Kk_X },\quad  \bB_E : = \bB_{\Kk_E },
 $$
 corresponding to these orbifold atlases have completions to 
 ep groupoids    $\bG_E , \bG_X $.  It follows from the construction
 that  the projection $\pr$ induces a functor
 $\pr:\bG_E \to \bG_X $ that restricts on the object spaces to the bundle projection $
 \bigsqcup_I E|_{W_I}\to  \bigsqcup_I {W_I}$.

 By  \cite[Proposition~4.19]{Mbr},\footnote{
 This result concerns the effective case, but 
applies equally well to the noneffective case because
 each groupoid has an effective quotient; see \cite[Def~2.33]{ALR}.  However, in \cite{Mbr} we took the fundamental class 
 of $\bG$ to be that of its effective quotient, while here we use the 
 more correct version that also takes into account the order of the group that acts noneffectively.}  
  one
 way to define the Euler class of $\pi: E\to X$ is to consider a \lq\lq nonsingular  resolution" of the groupoid $\bG_X $, 
 pull the bundle $E\to X$ back to this resolution and then push forward to $X$ the (weighted) zero set of a section $\nu$ of this bundle
 that is transverse to $0$ (written $\nu\pitchfork 0$). 
 As we explained above, 
we can take the resolution  of  $\bG_X $ to be the wnb groupoid $\bV_X $ formed as in Proposition~\ref{prop:V}
from a reduction of $\bG_X $.
The pullback of  $\pr:\bG_E \to \bG_X $ by $\io_\Vv: \bV_X \to \bG_X$ is the corresponding wnb groupoid with objects 
$ \bigsqcup_I E|_{V_I }$.  Let $\nu: \bV_X\to \bE_X$ be a section of this bundle.  This is given
 by a compatible family of sections 
\begin{equation}\label{eq:nuXE}
\nu_I: V_I \to E|_{V_I },\quad 
\nu_J|_{\TV_{IJ} } = \nu_I\circ \rho^X_{IJ}.
\end{equation}
If $\nu\pitchfork 0$, there is a full subcategory  $\bZ^\nu_X $ of  $\bV_X $
whose  objects $\nu_I^{-1}(0)\subset V_I $ form a closed $d$-dimensional submanifold $\Obj_{\bZ^\nu_X}$ of $\Obj_{\bV_X}$ of codimension equal to 
the fiber dimension of $E$.
 It is not hard to check that this has the structure of a wnb groupoid $\bZ^\nu_X$ with 
 the induced weighting function $\La_Z$ equal to the restriction of $\La_V$ to the image of the inclusion
 $|\bZ^\nu_X|_\Hh\to |\bV_X|_\Hh$. 
 
 The following is  a version of results proved in \cite[\S3]{Mbr}; see also \cite[\S5.2]{Mcn}.
 
 \begin{lemma} 
 Let $E\to X$ be an oriented orbibundle  and $d: = \dim X-\dim E$.
 Then the cobordism class of the wnb groupoid $(\bZ^\nu_X,\La_Z)$ constructed above is independent of choices,
 as is the image in $H_d(X;\Q)$ of the pushforward of its fundamental class .
 \end{lemma}

\begin{example} \rm   Consider the  football $X$ considered in Example~\ref{ex:foot} 
with  reduction as in Example~\ref{ex:foot3}.  Its tangent bundle $\rT X$ has a corresponding atlas with  charts 
$(\rT W_I, \Ga_I, \psi^{\rT X}_I)$ and reduction $\rT \Vv$ with domains $\rT V_I$.  
Trivialize the bundle $\rT V_{12}\to V_{12}$ by choosing a nonvanishing $\Ga_{12}$-invariant section $\nu_{2}$.  This descends to a nonvanishing section of $\rT V_i|_{V_{i(12)}}$ (where $V_{i(12)} : = \rho_{i(12)}(\TV_{i(12)}) $). Since each $V_i$ is a disc, for each $i$, this section extends to a section $\nu_i: V_i\to \rT V_i$ with precisely one zero, which has weight $\frac 1 {|\Ga_i|}$  Hence the Euler class is represented by the zero dimensional branched manifold that is represented by two points, one with weight $\frac 12$ and one with weight $\frac 13$.\MS
 \hfill$\er$
 \end{example}
 
 \begin{rmk}\label{rmk:polyf}
 \rm 
 This abstract method should also apply to the infinite dimensional orbi\-bundles 
of polyfold theory~\cite{HWZ}.  Here one has an orbibundle 
whose base and total space are $sc$-Banach manifolds.  Since the moduli space $X$ of $J$-holomorphic stable maps is compact,
one can define atlases $\bB_{\Kk_X}, \bB_{\Kk_E}$ as above that are finite (i.e. have  finitely many basic charts)
and  such that $|\bB_{\Kk_X}|$ is a neighbourhood of $X$. In particular, 
the projection 
is the realization of  a functor $\pi:\bB_{\Kk_E}\to \bB_{\Kk_X}$ 
that restricts on each chart to a bundle $\pi:E_I\to U_I$ with infinite dimensional base and fibers on which the finite group $
\Ga_I: = \prod_{i\in I} \Ga_i$ acts.
  We are also given a canonical smooth section $\s: = (s_I)$ where each $s_I: U_I\to E_I$ is  a $\Ga_I$-equivariant {\it Fredholm} operator such that the realization $|\s^{-1}(0)|$ of the zero set is canonically identified with $X$.    We can choose a subgroupoid $\bV_{\Kk_X}$ of $\bG_{\Kk_X}$ as in \eqref{eq:Vv}. Then polyfold Fredholm theory  
  implies that there are single valued sections $\nu$ of the pullback bundle such that $\s|_\bV + \nu\pitchfork 0$.
  The resulting zero set $\bZ^\nu$ has domains that are $d$-dimensional manifolds, where $d$ is the Fredholm index of $\s$, and just as above is a nonsingular \'etale groupoid whose realization has a natural  weighting function.
  The proof sketched above (and given in detail in \cite{MW3}) 
  that  $\bZ^\nu$  is a weighted branched manifold relies on the existence of a similar structure of the ambient groupoid $\bV_{\Kk_X}$.  
  In the polyfold setup, $\bV_{\Kk_X}$ is infinite dimensional.    Hence, in order to complete the proof that the zero set is
  a weighted branched manifold of dimension $d$ one would have to carefully check the properties of the local branching structure of the zero set.
  However, since this is entirely controlled by the behavior of the group actions, this should pose no problem, hence giving a simple model for the virtual cycles constructed in polyfold theory.
  We hope to return to this question in the future.
   \hfill$\er$
 \end{rmk}

\NI{\bf Acknowledgements:}  I would like to thank Dominic Joyce, Eugene Lerman and John Pardon for useful discussions about orbifolds.  This paper also builds on the insights of Katrin Wehrheim and the work in our joint project on Kuranishi atlases.  I also thank the referees for their careful reading and  helpful suggestions.

\end{document}